\documentclass[a4paper]{article}

\usepackage[left=35mm,right=35mm,top=30mm,bottom=35mm]{geometry}

\usepackage[utf8]{inputenc}
\usepackage{microtype}
\usepackage{csquotes}

\usepackage{enumitem}
\setlist[enumerate,1]{label={(\roman*)}}

\usepackage{mathtools}
\usepackage{dsfont}

\usepackage{amsmath}
\usepackage{amsfonts}
\usepackage{amssymb}
\usepackage{tikz}
\usetikzlibrary{patterns}
\usetikzlibrary{plotmarks}
\usetikzlibrary{positioning}

\usepackage{tabto}
\usepackage{stmaryrd}
\usepackage[percent]{overpic}
\usepackage{rotating}

\usepackage[nameinlink]{cleveref}
\crefname{eexample}{example}{examples}

\usepackage{braket}
\usepackage{bbm}
\usepackage{xcolor}
\usepackage{graphicx}
\usepackage{hhline}
\usepackage{algorithmicx}
\usepackage{algorithm}
\usepackage{accents}
\usepackage{amsthm}

\usepackage{enumitem}
\setlist[enumerate,1]{label=(\roman*),parsep=0.4\baselineskip}
\usepackage{moreenum}

\usepackage[T1]{fontenc}

\theoremstyle{plain}
\newtheorem{theorem}{Theorem}%[chapter]
\newtheorem{proposition}[theorem]{Proposition}
\newtheorem{corollary}[theorem]{Corollary}

\theoremstyle{definition}
\newtheorem{definition}[theorem]{Definition}
\newtheorem{remark}[theorem]{Remark}
\newtheorem{example}[theorem]{Example}
\newcommand{\norm}[1]{\left\lVert#1\right\rVert}
\DeclareMathOperator*{\argmin}{arg\,min}
\newcommand{\ubar}[1]{\underaccent{\bar}{#1}}

\begin{document}

\title{\bf Towards multiobjective optimization \\ and control of smart grids\thanks{The authors acknowledge funding by the Federal Ministry for Education and Research (BMBF; grant 05M18SIA). Karl Worthmann is indebted to the German Research Foundation (DFG; grant WO 2056/6-1).}}

\author{Philipp Sauerteig\thanks{Institute for Mathematics, Faculty of Mathematics and Natural Sciences, Technische Universit\"at Ilmenau, Ilmenau, Germany (philipp.sauerteig@tu-ilmenau.de).} \ and Karl Worthmann\thanks{Institute for Mathematics, Faculty of Mathematics and Natural Sciences, Technische Universit\"at Ilmenau, Ilmenau, Germany (karl.worthmann@tu-ilmenau.de).}}

\maketitle

\begin{center}
\begin{minipage}{0.8\textwidth}
	The rapid uptake of renewable energy sources in the electricity grid leads to a demand in load shaping and flexibility. Energy storage devices such as batteries are a key element to provide solutions to these tasks. However, typically a trade-off between the performance related goal of load shaping and the objective of having flexibility in store for auxiliary services, which is for example linked to robustness and resilience of the grid, can be observed. We propose to make use of the concept of Pareto optimality in order to resolve this issue in a multiobjective framework. In particular, we analyse the Pareto frontier and quantify the trade-off between the non-aligned objectives to properly balance them.
\end{minipage}
\end{center}

\section{Introduction}

The energy transition has triggered a rapid uptake of renewable energies in the electricity grid, which may lead to severe problems in the energy supply, see, e.g.\ the introductory paragraphs of~\cite{RatnWell13,WortKell15} and the references therein. In this paper, we focus on the integration of energy storage devices~\cite{WestNico08}, e.g.\ batteries, in microgrids~\cite{OlivMehr14}, see also~\cite{ParhLotf15} for a review article and~\cite{LotfKhod17} for questions concerning AC/DC. Based on forecasts~\cite{AppiOrdi18,MeerShep18}, consumption and production peaks can be anticipated, such that a receding horizon strategy, typically realized via a Model Predictive Control (MPC) scheme~\cite{PalmBena13,PariRiko14,worthmann2014distributed}, is idealy suited to tackle this task, see also the textbooks~\cite{GrunPann17,RawlingsMayneDiehl2017} for an introduction to (nonlinear) MPC. An aspect of particular interest is whether the optimization based control algorithm is realized in a decentralized, distributed, or centralized fashion, see, e.g.\ \cite{WortKell15} and \cite{HidaMyrz18}.

A major concern from a grid operator's perspective is the volatile power demand due to residential energy generation. In recent papers, the authors exploit inherent flexibility provided by energy storage devices to smoothen the aggregated power demand profile by approximating a given reference value~\cite{WortKell15}. Here, a key issue was the realization of the global optimum by using distributed control --- both in a cooperative~\cite{BrauGrun16} and a non-cooperative setting~\cite{BrauGrun17}. Another --~potentially conflicting~-- objective is to stay within time-varying tubes as introduced in~\cite{BrauFaul18}, which may not contain the reference value enforced by the first optimization goal. %
The motivation behind this approach is to provide flexibility to the grid operator, which may be used to counteract other shortcomings within the network~\cite{MajzKhod17}. The contribution of this paper is to embed this problem into a multiobjective framework, link it to the scalarized objective function and investigate the connection to the Pareto frontier. Furthermore, we also introduce a more restrictive optimality concept --~proper optimality in the sense of Geoffrion. By doing so, it becomes possible to quantify the trade-off between the different objectives, which is an essential prerequisite for structured decision making process. 

Moreover, we provide numerical simulations to investigate the MPC closed loop within our multiobjective framework. Hereby, we use data from an Australian distribution network dataset, 
see~\cite{RatnWell17} for further details.

{\bf Notation:} For two integers $m$ and $n$ with $m\leq n$, we use the notation $[m:n] := \{m,m+1,\ldots,n\}$.

\section{Problem Formulation}

In this section, we firstly introduce a mathematical model of a prosumer, i.e.\ a residential energy system equipped with both solar-photovoltaic panel (or any other device for energy generation) and an energy storage device like a battery. %
Secondly, we consider a (potentially) large network of $\mathcal{I}$, $\mathcal{I} \in \mathbb{N}$, such subsystems. The individual subsystems are connected via a common point of coupling, the grid operator (simply termed Central Entity; CE), see Figure~\ref{fig:u-constraints} (left). %
So far, we essentially follow the modelling approach presented in~\cite{RatnWell13,WortKell15} and further developed in~\cite{BrauFaul18}. An alternative would be to also model the underlying network structure as graph and, then, to incorporate the physical connections within the microgrid, see, e.g.\ \cite{FlocBans17}.%

The CE  
may pursue different objectives, which leads to a Multiobjective Optimization Problem (MOP) since we assume that the subsystems are cooperative. Here, we exemplarily present two criteria and numerically investigate the network performance w.r.t.\ each of the two objectives.

\subsection{Modelling of the System Dynamics and the Constraints}

The $i$-th subsystem, $i \in [1:\mathcal{I}]$, is given by
\begin{eqnarray}\label{eq:RES}
	\begin{split}
		x_i(k+1)&=&\alpha_ix_i(k)+T(\beta_iu_i^+(k)+u_i^-(k))\\
		z_i(k)&=&w_i(k)+u_i^+(k)+\gamma_iu_i^-(k)
	\end{split}
\end{eqnarray}
where $x_i(k)$ describes the State Of Charge (SOC) of the battery and $z_i(k)$ the power demand at time instant $k \in \mathbb{N}_0$. The parameter $T>0$ is the length of one time step in hours, e.g. $T=0.5$ corresponds to a half hour window. %
The variables $u_i^+(k)$ and $u_i^-(k)$ represent the charging and discharging rate, resp., while $w_i(k)$ is the net consumption (load minus generation). The constants $\alpha_i,\beta_i,\gamma_i\in(0,1]$ model efficiencies w.r.t. to self discharge and energy conversion, resp. %
The initial SOC at $k=0$ is denoted by $x_i^0$, i.e.\ $x_i(0) = x_i^0$ holds. % 
The system variables are subject to the constraints
\begin{subequations}\label{subeq:constraints}
\begin{eqnarray}
	0 \quad \leq \quad & x_i(k) & \leq \quad C_i \\
	\label{eq:constraints:discharge} \underline{u}_i \quad \leq \quad & u_i^-(k) & \leq \quad 0\\
	\label{eq:constraints:charge} 0 \quad \leq \quad & u_i^+(k) & \leq \quad \bar{u}_i \\
	\label{eq:constraints:chargeANDdischarge} 0 \quad \leq \quad & \frac{u_i^-(k)}{\underline{u}_i} + \frac{u_i^+(k)}{\bar{u}_i} \quad & \leq \quad 1,
\end{eqnarray}
\end{subequations}
where $C_i \geq 0$ denotes the battery capacity . %
Constraint~\eqref{eq:constraints:chargeANDdischarge} ensures that the bounds on charging~\eqref{eq:constraints:charge} and discharging~\eqref{eq:constraints:discharge} hold, even if charging and discharging occur during the same time interval, see Figure~\ref{fig:u-constraints} (right). For a concise notation, we define the state constraint $\mathbb{X}_i := [0,C_i]$ and the set of feasible control values $\mathbb{U}_i := \Set{(u_i^-,u_i^+)^\top \in \mathbb{R}^2| \text{{\eqref{eq:constraints:discharge}-\eqref{eq:constraints:chargeANDdischarge} hold}}}$. Furthermore, we use the notation $\bar{\cdot}$ to denote the average value over all subsystems, e.g. $\bar{w}(k) = \frac{1}{\mathcal{I}} \sum_{i=1}^{\mathcal{I}} w_i(k)$.
\begin{figure}[htb]
	\begin{center}
		\begin{overpic}[width=.35\textwidth]{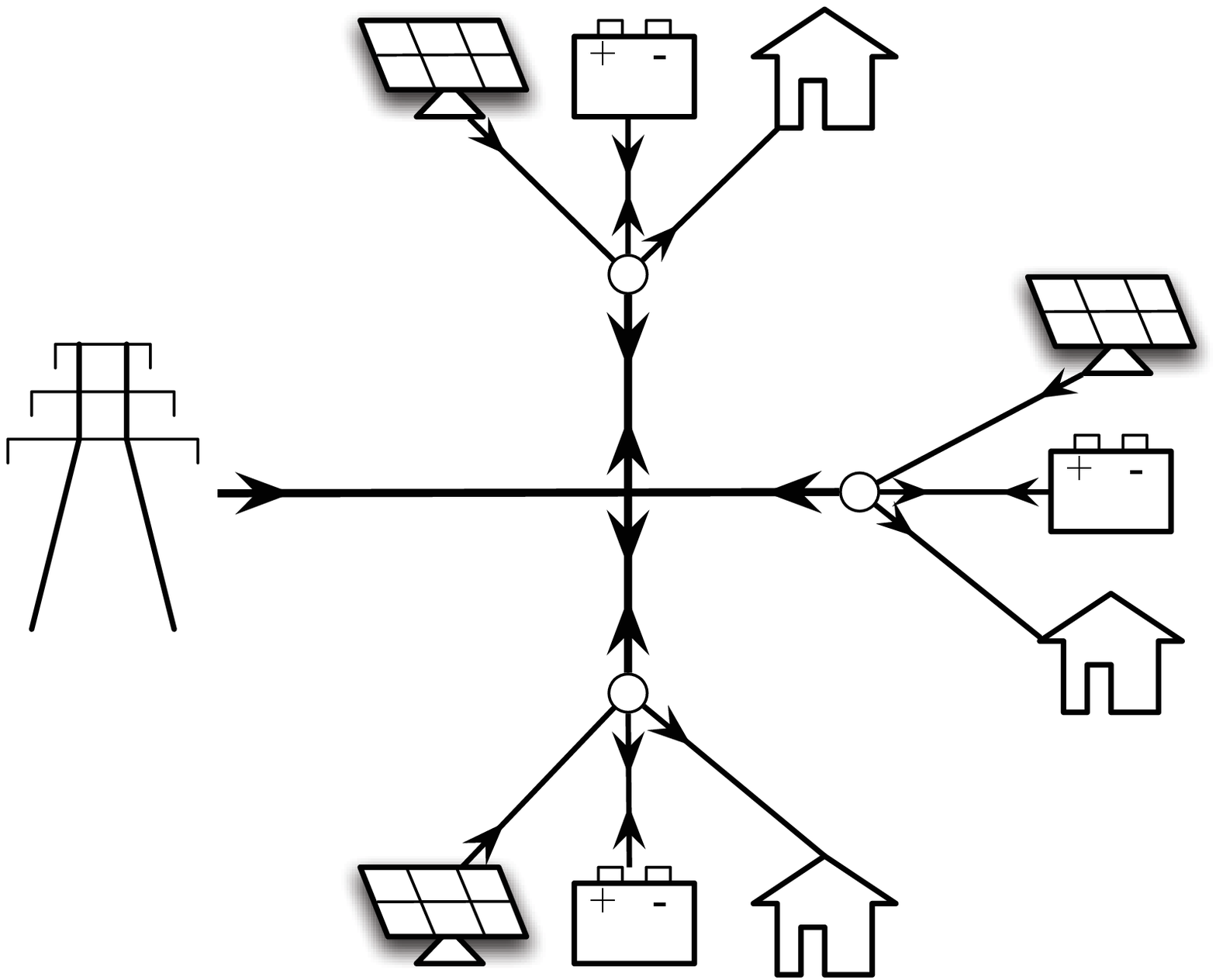}%
			\put(4,23){{\bf \scriptsize C.E.}}%
			\put(48.5,80){\scriptsize RES$_1$}%
			\put(83,60){{\scriptsize RES$_2$}}%
			\put(77,9){\scriptsize \begin{sideways}$\ddots$\end{sideways}}%
			\put(72.5,0.5){\scriptsize RES$_\mathcal{I}$}%
		\end{overpic}\hspace*{0.2\textwidth}%
		\includegraphics[scale=0.85]{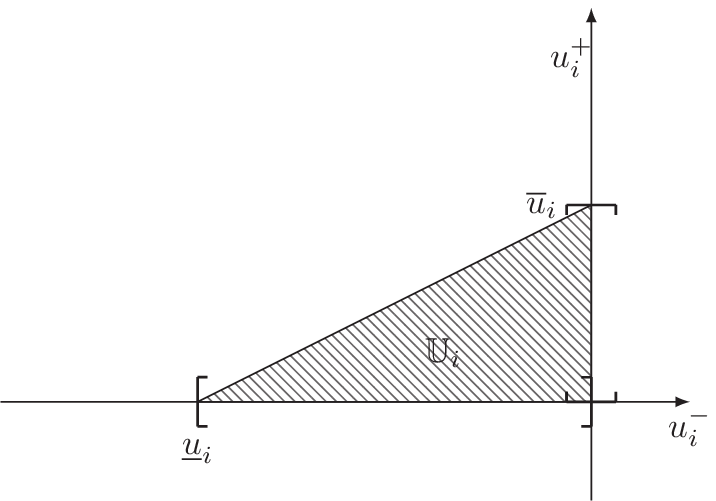}
	\caption{Schematic representation of a network of Residential Energy Systems (left) and the set of feasible control values (right).}
	\label{fig:u-constraints}
	\end{center}
\end{figure}

\subsection{First Objective: Peak Shaving}

The Central Entity is located at the common point of coupling, where the aggregated sum $\sum_{i=1}^{\mathcal{I}} z_i(k)$ is collected. A positive sign corresponds to a power demand while a negative sign indicates a surplus, which reflects that the aggregated local generation exceeds the aggregated local load at time instant~$k$. %, $k \in \mathbb{N}_0$. 
Hence, the CE requires balancing energy to match the aggregated power demand. While that is comparatively easy for a constant power demand/supply, it is significantly more demanding %(and more expensive) 
if the power demand exhibits large fluctuations. Hence, a typical goal is to flatten the energy demand %/supply 
of the network in consideration.

For a mathematical formulation, we first compose an overall system by setting up the system dynamics
$$
	x(k+1) = \begin{pmatrix}
		x_1(k+1) \\ x_2(k+1) \\ \vdots \\ x_{\mathcal{I}}(k+1)
	\end{pmatrix} = \begin{pmatrix}
		f_1(x_1(k),u_1(k)) \\
		f_2(x(_2(k),u_2(k))\\
		\vdots\\
		f_{\mathcal{I}}(x_{\mathcal{I}}(k),u_{\mathcal{I}}(k))
	\end{pmatrix} = f(x(k),u(k)),
$$
the constrains $x(k) \in \mathbb{X} := \mathbb{X}_1 \times \ldots \times \mathbb{X}_{\mathcal{I}} \subset \mathbb{R}^\mathcal{I}$ and $u(k) \in \mathbb{U} := \mathbb{U}_1 \times \ldots \times \mathbb{U}_{\mathcal{I}} \subset \mathbb{R}^{2\mathcal{I}}$. Hence, for given data
\begin{align}
	w(n) = (w_1(n), w_2(n), \ldots, w_{\mathcal{I}}(n))^\top, \qquad n = 0,1,2,\ldots,k, \label{NotationLoadData}
\end{align}
the power demand at time $k \geq 0$ resulting from a discharging/charging behaviour $(u(n))_{n=0}^{k} = ((u^-(n),u^+(n)))_{n=0}^{k} \in \mathbb{U}^{k+1}$ and initial state of charge $x^0 \in \mathbb{X}$ is
\begin{align}
	z(k) = z(k;(u(n))_{n=0}^{k},(w(n))_{n=0}^k,x^0). \nonumber
\end{align}
Here, we emphasize %like to point out 
that the output depends on the time varying load and generation data given by \eqref{NotationLoadData}, which has to be predicted for load shaping, see, e.g.\ \cite{BaliGaur15,KhanMahm16} for load forecasting, \cite{MakaEtin11,WanXu14} and \cite{GolePier16,AntoOsor16} for predictions on the energy generation due to wind and solar power, resp. Hence, at time instant $k$, %\in \mathbb{N}_0$, 
we optimize over a finite time horizon of $N$, $N \in \mathbb{N}_{\geq 2}$, time steps. Moreover, we require suitable (time varying) reference values, which are constructed based on the overall average net consumption, i.e.
\begin{align}\label{NotationReferenceValues}
	\bar{\zeta}(n) = \frac {1}{\mathcal{I} \cdot \min\{N,n+1\}} \sum_{j = n - \min\{n,N-1\}}^n \sum_{i=1}^{\mathcal{I}} w_i(j).
\end{align}
Note that $\bar{\zeta}(n)$ is the average over the previous $N$ time steps if the data history is sufficiently long, which is reflected by taking the minimum. Then, for given data~$w(n)$, $n = \max\{0,k-(N-1)\}, \ldots, k, \ldots, k+N-1$, and state~$x(k)$, $x(k) \in \mathbb{R}^{\mathcal{I}}$, the CE wants to minimize the objective function $J_1:\mathbb{U}^{N} \to \mathbb{R}_{\geq0}$,
\begin{align*}
	\mathbf{u} = (u(n))_{n=k}^{k+N-1} \quad\mapsto\quad \frac {1}{N} \sum_{n=k}^{k+N-1} \left( \frac {1}{\mathcal{I}} \sum_{i=1}^{\mathcal{I}} \left[w_i(n)+u_i^+(n)+\gamma_iu_i^-(n)\right]-\bar{\zeta}(n)\right)^2,
\end{align*}
subject to the system dynamics~\eqref{eq:RES} and the state constraint $x_i(n) \in \mathbb{X}_i$ for all $n \in [k:k+N]$ and $i \in [1:\mathcal{I}]$, which penalizes the deviation of the average power demand from the sequence of reference values given by the vector $\bar{\zeta} = (\bar{\zeta}(k),\ldots,\bar{\zeta}(k+N-1))^\top$. Here, we call $N$ the optimization horizon, on which it is assumed that the load and generation data can be reliably predicted.

The optimal control value of this optimization problem is unique and is attained as shown in \cite{BrauGrun16}. Hence, we can define a state feedback law by using the first element $u^\star(k)$ of an optimal control sequence $\mathbf{u}^\star = (u^\star(n))_{n=k}^{k+N-1}$, which is then implemented at the plant. Then, based on the new state $x(k+1)$ and the new input data (in particular $w(k+N)$ and $\bar{\zeta}(k+N)$), the procedure can be repeated ad infinitum to generate a closed loop in a receding horizon fashion as explicated in Algorithm~\ref{alg:MPC} solely based on $J_1$ (and, thus, ignoring the --~potentially conflicting~-- second objective~$J_2$, see, e.g.\ \cite{WortKell15} or \cite{GrunPann17} for a more detailed %general 
introduction to (distributed) MPC). %

Numerical results illustrating the outcome of such an MPC scheme are depicted in Figure~\ref{fig:peak_shaving}\hspace{-2mm}. 
We point out that the reference values \eqref{NotationReferenceValues} cannot be traced due to the constraints~\eqref{subeq:constraints}.
\begin{figure}[htb]
	\centering
	\includegraphics[width=0.95\textwidth]{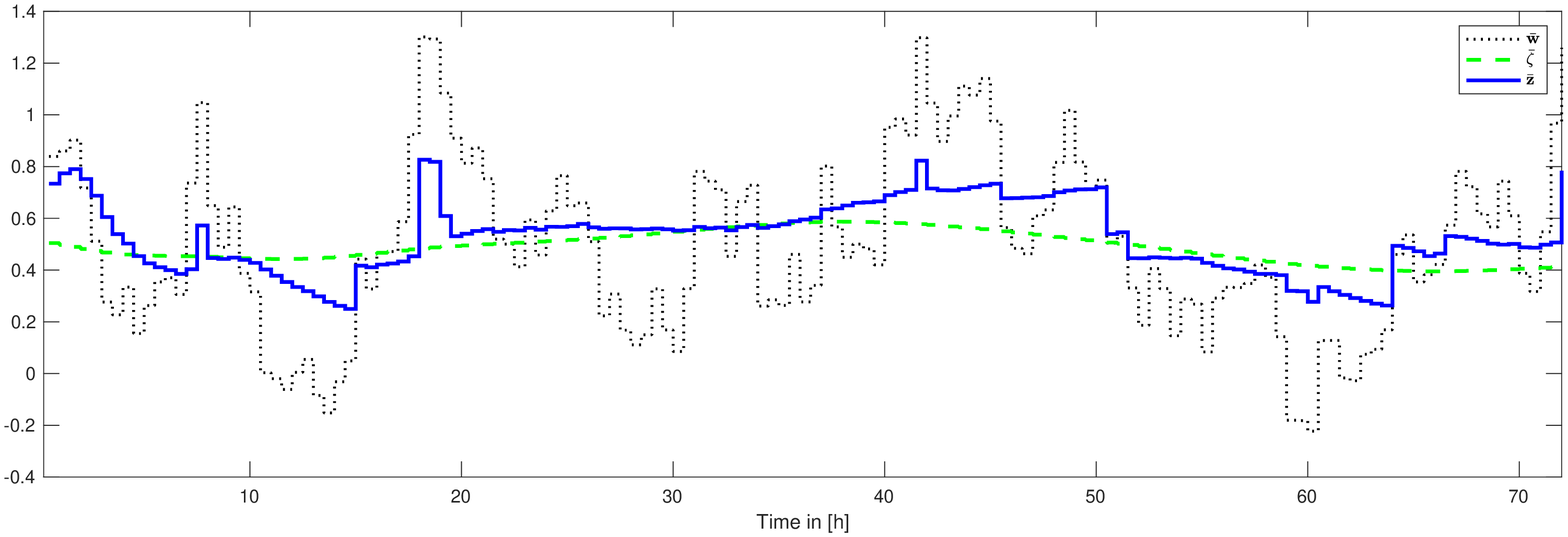}\\
	\includegraphics[width=0.95\textwidth]{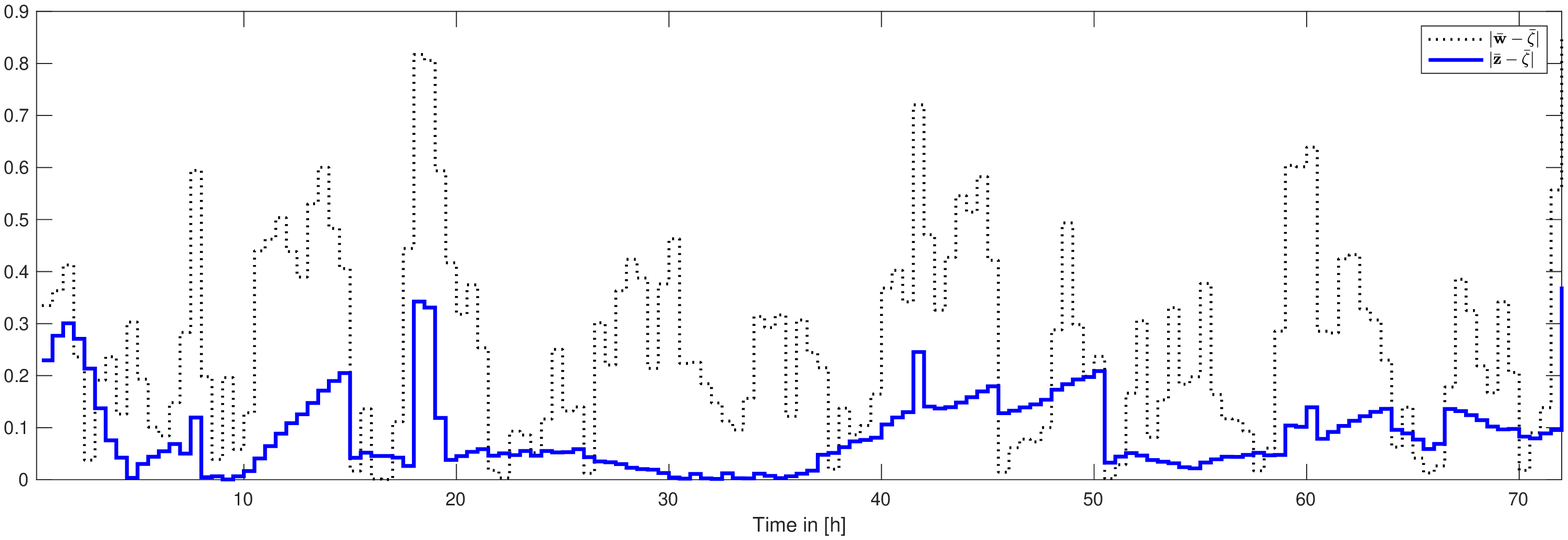}
	\caption{Average power demand $\frac 1 {\mathcal{I}} \sum_{i=1}^{\mathcal{I}} z_i(k)$ (top) and its (absolute) deviation from the reference $\bar{\zeta}(k)$ (bottom) of the MPC closed loop with prediction horizon $N = 48$ (one day) w.r.t.\ peak shaving in comparison to the setting without energy storage.}
	\label{fig:peak_shaving}
\end{figure}

\subsection{Second Objective: Flexibility via Tube Constraints}

Besides flattening the power demand profile, another goal of the grid operator may be to provide flexibility to the lower/higher layers of the overall system. 
The idea is to introduce time-varying tube constraints on the average power demand, i.e., we would like to achieve
\begin{align}\label{NotationTubeConstraints1}
	\ubar{c}(n) \quad \leq \quad \frac 1 {\mathcal{I}} \sum_{i=1}^{\mathcal{I}} z_i(n) \quad \leq \quad \bar{c}(n) \qquad\forall\,n \in [k:k+N-1]
\end{align}
for the lower and upper bounds $\ubar{\mathbf{c}}(n)$ and $\bar{\mathbf{c}}(n)$, $n \in [k:k+N-1]$. However, since the inclusion of the tube constraints~\eqref{NotationTubeConstraints1} may render the problem infeasible, we minimize the violation of the tube constraints by using the objective function $J_2: \mathbb{U}^{N} \rightarrow \mathbb{R}_{\geq 0}$,
\begin{align}\nonumber
%	\mathbf{u} = (u(n))_{n=k}^{k+N-1} \quad\mapsto\quad 
	\mathbf{u} \quad \mapsto \quad \sum_{n=k}^{k+N-1} \left[ \left( \max\Big\{0,\frac {1}{\mathcal{I}} \sum_{i=1}^{\mathcal{I}} z_i(n) - \bar{c}(n)\Big\} \right)^2 + %\sum_{n=k}^{k+N-1} 
	\left( \max\Big\{0, \ubar{c}(n) - \frac {1}{\mathcal{I}} \sum_{i=1}^{\mathcal{I}} z_i(n)\Big\} \right)^2 \right],
\end{align}
subject to the system dynamics~\eqref{eq:RES} and the state constraint $x_i(n) \in \mathbb{X}_i$
for all $n \in [k:k+N]$ and $i \in [1:\mathcal{I}]$. According to this construction a positive value of $J_2(\mathbf{u})$ reflects a violation of~\eqref{NotationTubeConstraints1} which causes a loss of flexibility in the system.

Note that this cost function exhibits some complementarity: for a given time index~$n$, strict positivity of one term implies that the other term is equal to zero, which reflects that a violation of Condition~\eqref{NotationTubeConstraints1} can only occur on one side. Moreover, for a given time index~$n$, the objective function is convex. Numerical solutions are displayed in Figure~\ref{fig:tube_constraints} to illustrate our modelling approach.
\begin{figure}[htb]
	\centering
	\includegraphics[width=0.95\textwidth]{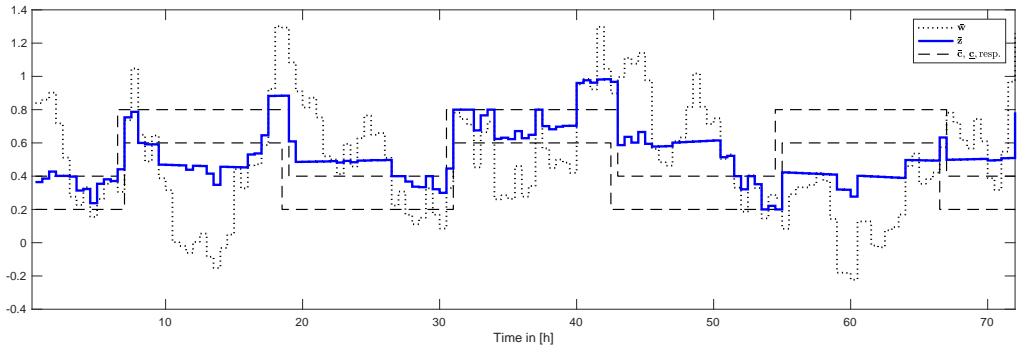}
	\caption{Evolution of the average power demand $\frac 1 {\mathcal{I}} \sum_{i=1}^{\mathcal{I}} z_i(k)$ 
	and the time varying tube constraints \eqref{NotationTubeConstraints}.}
	\label{fig:tube_constraints}
\end{figure}

\subsection{Multiobjective Optimization Problem and MPC Scheme}\label{subsec:MOPandMPC}

Combining the non-aligned (conflicting) objectives introduced in the preceeding subsections leads to the following multiobjective optimization problem
		\begin{align*}
			\min_{\mathbf{u} \in \mathbb{U}^N} \quad & \begin{pmatrix}
				J_1(\mathbf{u}) \\
				J_2(\mathbf{u})
			\end{pmatrix} \qquad\text{subject to \eqref{eq:RES} and $x_i(n) \in \mathbb{X}_i$ for all $(i,n) \in [1:\mathcal{I}] \times [k:k+N]$.}
		\end{align*}
For this kind of Multiobjective Optimization Problems (MOPs), see, e.g.\ \cite{Jahn04}, the following optimality concept is typically used.
\begin{definition}[Pareto optimality, Pareto frontier]\label{def:pareto}
	Consider the vector optimization problem
	\begin{equation}\label{OP}
		\min_{x\in\mathcal{M}}\quad f(x)
	\end{equation}
	with $f : X \to \mathbb{R}^m$, $\mathcal{M} \subseteq X \subseteq \mathbb{R}^n$. %
	A vector $x^\star \in \mathcal{M}$ is called \emph{(Pareto) optimal} or \emph{(Pareto) efficient} for the optimization problem~\eqref{OP} %
	if for all $x\in\mathcal{M}$ and $i\in\{1,\ldots,m\}$ the 	following holds:
	$$
		f_i(x)<f_i(x^\star) \quad\Longrightarrow\quad \exists\,j\in\{1,\ldots,m\}: f_j(x^\star)<f_j(x),
	$$
	where $f_i : X \to \mathbb{R}$ denotes the $i$-th component of $f$, $i \in [1:m]$.
	Furthermore, a point $x^\star \in \mathcal{M}$ is called \emph{weakly (Pareto) optimal} or \emph{weakly (Pareto) efficient} of~\eqref{OP} if there is no $x \in \mathcal{M}$ with $f_i(x) < f_i(x^\star)$ for all $i \in [1:m]$. 
	The set $\Set{f(x) | x \text{ is efficient}}$ is called \emph{Pareto frontier}.
\end{definition}

\begin{remark}
Clearly, efficiency yields weak efficiency.
\end{remark}

Before we proceed, we reformulate the MOP described above in order to facilitate its solution in Section~\ref{SectionScalarizationAndADMM}. To this end, we begin with the objective function~$J_2$. Here, we introduce the stacked auxiliary variable $\mathbf{s}=(\underline{\mathbf{s}}^\top,\bar{\mathbf{s}}^\top)^\top \in \mathbb{R}_{\geq0}^{2N}$ with
\begin{align}\nonumber
	\underline{\mathbf{s}} = (\underline{s}(k),\underline{s}(k+1),\ldots,\underline{s}(k+N-1))^\top \qquad\text{ and }\qquad \bar{\mathbf{s}} = (\bar{s}(k),\bar{s}(k+1),\ldots,\bar{s}(k+N-1))^\top,
\end{align}
which was already used in \cite{BrauFaul18}. Then, we may introduce the relaxed constraints
\begin{align}\label{NotationTubeConstraints}
	\ubar{c}(n) - \ubar{s}(n) \quad\leq\quad \frac 1 {\mathcal{I}} \sum_{i=1}^{\mathcal{I}} z_i(n) \quad\leq\quad \bar{c}(n) + \bar{s}(n) \qquad\text{ for all $n \in [k:k+N-1]$}
\end{align}
and minimize the Euclidean norm of the vector $\mathbf{s}$. To this end, we introduce the cost function $h : \mathbb{R}^{2N}_{\geq 0} \to \mathbb{R}_{\geq 0}$,
\begin{align}\nonumber
	\mathbf{s} \quad\mapsto\quad \norm{\mathbf{s}}_2^2 = \sum_{n=k}^{k+N-1} \underline{s}(n)^2 + \bar{s}(n)^2.
\end{align}
In summary, the additional signed (slack) variables, which are stacked in the vector~$\mathbf{s}$, allow to represent the optimization objective~$J_2$ by the strictly convex function~$h$ at the expense of incorporating the constraints given by~\eqref{NotationTubeConstraints}. Doing so ensures feasibility of the optimization problem, e.g.\ by not using the storage devices at all. If desired from an implementational point of view, the complementarity constraint $\underline{s}(n) \bar{s}(n) = 0$, $n \in [k:k+N-1]$, may be included while both feasibility and the value of the objective function are maintained.

Furthermore, we introduce the auxiliary variable $\bar{\mathbf{z}} = (\bar{z}(k),\ldots,\bar{z}(k+N-1))^\top \in \mathbb{R}^N$ and the additional equality constraints
\begin{align}\label{NotationZbar}
	\bar{z}(n) = \frac 1 {\mathcal{I}} \sum_{i=1}^{\mathcal{I}} z_i(n) \qquad\text{ for all $n \in [k:k+N-1]$}.
\end{align}
Using $\bar{\mathbf{z}}$ and the constraint~\eqref{NotationZbar} allows us to replace $J_1$ by the function $g: \mathbb{R}^N \rightarrow \mathbb{R}_{\geq 0}$,
\begin{align}\nonumber
	\bar{\mathbf{z}} \quad\mapsto\quad \frac 1N \sum_{n=k}^{k+N-1} (\bar{z}(n) - \bar{\zeta}(n))^2 = \frac 1N \norm{\bar{\mathbf{z}} - \bar{\zeta}}_2^2.
\end{align}
Next, we want to rephrase the constraint set such that the \textit{output} variables $\mathbf{z}_i = (z_i(k),z_i(k+1),\ldots,z_i(k+N-1))^\top$, $i \in [1:\mathcal{I}]$, can be used as an optimization variable instead of $\mathbf{u}$. To this end, we define the set $\mathbb{D} = \mathbb{D}_1 \times \ldots \times \mathbb{D}_{\mathcal{I}}$ with
\begin{align}\nonumber
	\mathbb{D}_i & = \mathbb{D}_i(x_i(k),(w_i(n))_{n=k}^{k+N-1})  \\
	& := \left\{ 
		\mathbf{z}_i \in \mathbb{R}^N \left| \exists\, \mathbf{u} \in \mathbb{U}^N: \begin{matrix}
			x_i(n+1) = \alpha_i x_i(n) + T(\beta_i u_i^+(n) + u_i^-(n)) \in [0,C_i],\\
			z_i(n) = w_i(n) + u_i^+(n) + \gamma_i u_i^-(n),\,n \in [k:k+N-1]\\
		\end{matrix} \right. \right\} \nonumber
\end{align}
with $x_i(k) \in [0,C_i]$ for all $i \in [1:\mathcal{I}]$. Note that, for all $i \in [1:\mathcal{I}]$, the set $\mathbb{D}_i$ and, thus, also the set $\mathbb{D}$ is convex and compact in view of the linearity of the constraints~\eqref{eq:RES} and~\eqref{subeq:constraints}. 

Then, plugging $\bar{z}$ into the constraint~\eqref{NotationTubeConstraints} yields the following MOP
\begin{equation}\label{MOP}\tag{MOP}\fbox{\centering
	\begin{minipage}{0.6\textwidth}
		\begin{align*}
			\min_{(\bar{\mathbf{z}}, \mathbf{s})} \quad & \begin{pmatrix}
				g(\bar{\mathbf{z}}) \\
				h(\mathbf{s})
			\end{pmatrix}\\
			\mathrm{subject \; to} \quad &(\bar{\mathbf{z}}, \mathbf{s}) \in \mathbb{S}=\Set{(\bar{\mathbf{z}},\mathbf{s}) \in \mathbb{R}^N \times \mathbb{R}^{2N}_{\geq0}| \left( \begin{array}{r}
				I \\-I
			\end{array} \right)
			\bar{\mathbf{z}}-\mathbf{s}\leq \left( \begin{array}{r}
				\bar{\mathbf{c}} \\ -\underline{\mathbf{c}}
			\end{array} \right) }\\
			& \bar{\mathbf{z}} : \exists\, (\mathbf{z}_1^\top, \ldots, \mathbf{z}_{\mathcal{I}}^\top)^\top \in \mathbb{D} \text{ satisfying \eqref{NotationZbar},}
		\end{align*}
	\end{minipage}}
\end{equation}
where the set $\mathbb{S}$ is convex. In \eqref{MOP} the optimization variables and, thus, the objectives are coupled via the constraint~$\mathbb{S}$. If, in addition, also individual objectives of the residential energy systems have to be taken into account, the vectors $\mathbf{z}_1,\ldots,\mathbf{z}_{\mathcal{I}}$ have to be used as optimization variables.

Finally, 
we explicate the Model Predictive Control (MPC) scheme to be applied.
\begin{algorithm}\caption{Model Predictive Control (MPC)}\label{alg:MPC}
	{\bf Input:} Prediction horizon~$N$. Set time $k = 0$.\\
	{\bf Repeat:}
		\begin{enumerate}
			\item [1.] Update the forecast data $w(n) \in \mathbb{R}^{\mathcal{I}}$, $n \in [\max\{k-N+1,0\} : k+N-1]$, compute $\bar{\zeta} = (\bar{\zeta}(k), \ldots, \bar{\zeta}(k+N-1))^\top$ and measure the current SOC $x_i(k)$, $i \in [1:\mathcal{I}]$.
			\item [2.] Determine an efficient point $(\bar{\mathbf{z}}^\star,\mathbf{s}^\star)$ and a corresponding input~$\mathbf{u}^\star$ such that $J_1(\mathbf{u}^\star) = g(\bar{\mathbf{z}}^\star)$ and $J_2(\mathbf{u}^\star) = h(\mathbf{s}^\star)$ hold.
			\item [3.] Implement $u^\star(k)$ at the plant, and increment the time index~$k$.
		\end{enumerate}
\end{algorithm}
In general, it is quite difficult to provide a meaningful analysis of the MPC closed loop resulting from a multiobjective problem formulation. However, adding a few additional assumptions allows to provide some guarantees if a stabilization task is considered, see, e.g.~\cite{GrunStie17}. In our analysis, however, we restrict ourselves to the analysis of the \textit{static} MOP to be solved in Step~2 of Algorithm~\ref{alg:MPC} and provide numerical simulations to investigate the MPC closed loop.

\section{Scalarization and Numerical Solution via ADMM}\label{SectionScalarizationAndADMM}

In the following, we consider the Scalarized Multiobjective Optimization Problem (SMOP) using the scalarization parameter~$\kappa$, $\kappa \in [0,1]$. For details on the scalarization of MOPs, we refer to \cite{Eich08}. Overall, we obtain the following SMOP.
\begin{equation}\label{SMOP}\tag{SMOP}\fbox{\centering
	\begin{minipage}{0.6\textwidth}
		\begin{align*}
			\min_{(\bar{\mathbf{z}},\mathbf{s})} \quad & \kappa g(\bar{\mathbf{z}}) + (1-\kappa) h(\mathbf{s}) \\
			\mathrm{subject \; to} \quad &(\bar{\mathbf{z}},\mathbf{s})\in\mathbb{S}=\Set{(\bar{\mathbf{z}},\mathbf{s}) \in \mathbb{R}^N \times \mathbb{R}^{2N}_{\geq0}|\begin{pmatrix}
				I\\-I
			\end{pmatrix}
			\bar{\mathbf{z}}-\mathbf{s}\leq\begin{pmatrix}
				\bar{\mathbf{c}}\\-\underline{\mathbf{c}}
			\end{pmatrix}}\\
			& \bar{\mathbf{z}} : \exists\, (\mathbf{z}_1^\top,\ldots,\mathbf{z}_{\mathcal{I}}^\top)^\top \in \mathbb{D} \text{ satisfying \eqref{NotationZbar}}
		\end{align*}
	\end{minipage}}
\end{equation}

In this section, we firstly present a numerical case study. Then, we briefly recall ADMM before we apply this optimization technique to solve~\eqref{SMOP}. ADMM is an algorithm for distributed optimization of convex, large-scale systems, which allows to solve~\eqref{SMOP} even for a very large number~$\mathcal{I}$ of subsystems. %
To this end, we exploit the reformulation of the original \eqref{MOP}. We show that the proposed algorithm yields, for each $\kappa \in [0,1]$, not only the optimal value of the scalarized problem but that the corresponding values of the components given by $g$ and $h$ are also unique. Then, based on the results derived in this section, we investigate the connection of the parameter~$\kappa$ to the Pareto frontier in the subsequent section.

\subsection{Numerical Simulations}

Increasing the value $\kappa\in[0,1]$ shifts the focus from penalizing the violation of the tube constraints to flatten the aggregated power demand profile. In Figure~\ref{fig:closedLoop} the corresponding closed-loop performance at time instant $k=0$ is visualized. Here, we set $\mathcal{I} = 10$, $T = 0.5$ (time step of $30$ minutes), $N = 48$ (one day prediction horizon), $x_i(0) = 0.5$, $C_i = 2$ (battery capacity), and $-\underline{u}_i = \bar{u}_i = 0.5$ (maximal discharging/charging rate) and vary the tube constraints from $(\underline{\mathbf{c}}_1,\bar{\mathbf{c}}_1)=(0.2,0.4)$ to $(\underline{\mathbf{c}}_2,\bar{\mathbf{c}}_2)=(0.6,0.8)$ on a time window of three days. %
Note that the choice of the tube constraints seems inappropriate since it is not aligned with the uncontrolled power demand profile (dotted black line in Figure~\ref{fig:closedLoop}). However, doing so, allows us to ensure that the objectives are potentially conflicting goals. Since weather forecasting is not topic of this paper we use real world data on load and generation provided by an Australian distribution network for the prediction of the future net consumption $w_i$, $i \in [1:\mathcal{I}]$. The choice of $T$ corresponds to this data. In~\cite{BrauFaul18} the authors chose a different cost function $g$ and simply summed the objectives, which complies with a weighting parameter $\kappa = 0.5$.

The corresponding open-loop perfomance, i.e. implementing $u^\star = (u^\star(k)^\top,\ldots,u^\star(k+N-1)^\top)^\top$ instead of $u^\star(k)$, can be viewd in Figure~\ref{fig:pareto} (top right).
\begin{figure}[htb]
	\centering
	\includegraphics[width=0.95\textwidth]{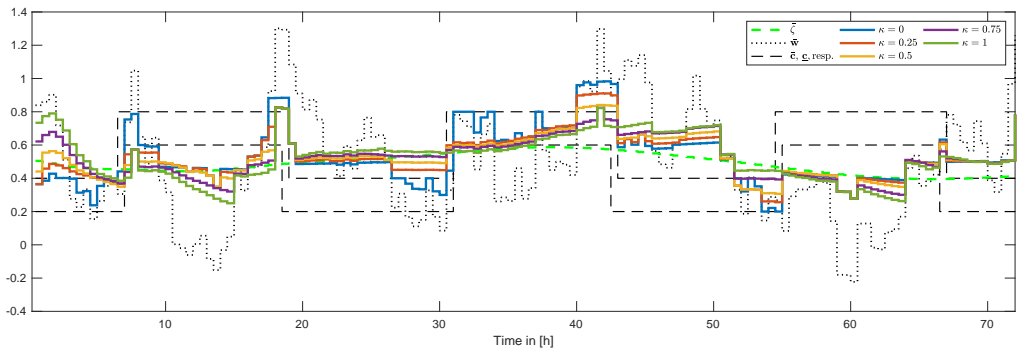}
	\caption{Impact of choice of weighting parameter $\kappa \in [0,1]$ on closed-loop performance.}
	\label{fig:closedLoop}
\end{figure}

\subsection{Alternating Direction Method of Multipliers (ADMM)}

There are two natural approaches to solve the optimization problem~\eqref{SMOP}. The first one is to calculate a centralized solution by optimizing the overall system at once. Here, a large rate of communication within the network is needed and the CE has to know all data of every prosumer, including possible charging rates, battery capacity, and current SOC, in detail. The second one renounces the communication by optimizing the behaviour of each prosumer separately. In this case the overall optimum cannot be guaranteed, see, e.g.\ \cite{WortKell15}. A remedy is given by distributed optimization algortihms like dual decomposition or ADMM, see, e.g.\ \cite{BrauSaue19} for further details.

ADMM is an algorithm used to solve large-scale optimization problems of the form
\begin{align*}
	\min_{(\mathbf{z},\mathbf{s})} \quad & \sum_{i=1}^{\mathcal{I}} f_i(\mathbf{z}_i)+g(\bar{\mathbf{z}}) + h(\mathbf{s}) \\
	\mathrm{subject \; to} \quad & (\bar{\mathbf{z}}, \mathbf{s}) \in \mathbb{S} \\
	& \bar{\mathbf{z}} = \frac{1}{\mathcal{I}} \sum_{i=1}^{\mathcal{I}} \mathbf{z}_i\\
	& \mathbf{z}_i \in \mathbb{D}_i \, \forall \, i \in [1:\mathcal{I}]
\end{align*}
in a distributed fashion. Here, the extended functions $f_i: D \to \mathbb{R} \cup \{ \pm \infty \}$, $i \in [1:\mathcal{I}]$, $g: D \to \mathbb{R} \cup \{ \pm \infty \}$, and $h: D \to \mathbb{R} \cup \{ \pm \infty \}$, defined on the closed and convex domain~$D$ are supposed to be proper, i.e.\ the effective domain $\Set{x \in D | f(x) \in \mathbb{R}}$ is nonempty and there does not exist $x \in D$ such that $f(x)=-\infty$ holds (the definition analogously holds for $g$ and $h$). 
\begin{remark}
	The functions $f_i$, $i \in [1:\mathcal{I}]$, may be used to incorporate individual goal of the subsystems in the optimization, see \cite{BrauFaul18} for a detailed explanation.
\end{remark}

Next, we outline the ADMM algorithm used in the following.
\begin{algorithm}
	\caption{ADMM}\label{alg:ADMM}
	{\bf Input:} $\mathbf{z}_i^0 \in \mathbb{R}^N$, $i \in [1:\mathcal{I}]$, $\Pi^0 \in \mathbb{R}^N$, and $\ell=0$.\\
	{\bf Repeat:}
		\begin{enumerate}
			\item[1.] {\bf Update subsystems}: For each $i\in[1:\mathcal{I}]$, compute
				$$
					\mathbf{z}_i^{\ell+1} \quad \in \quad \argmin_{\mathbf{z}_i \in \mathbb{D}_i} \quad \frac\rho2 \norm{\mathbf{z}_i - \mathbf{z}_i^\ell + \Pi^\ell}_2^2
				$$
				and send $\mathbf{z}_i^{\ell+1}$ to the CE.
			\item[2.] {\bf Update CE}: Calculate $\bar{\mathbf{z}}^{\ell+1} = \frac 1 {\mathcal{I}} \sum_{i=1}^{\mathcal{I}} \mathbf{z}_i^{\ell+1}$ and compute
				$$
					(\bar{\mathbf{a}}^{\ell+1}, \mathbf{s}^{\ell+1}) \quad \in \quad %
					\argmin_{(\bar{\mathbf{a}},\mathbf{s}) \in \mathbb{S}} \quad \kappa g(\bar{\mathbf{a}}) + (1-\kappa) h(\mathbf{s}) %
					+\frac{\rho \mathcal{I}}{2}\norm{\bar{\mathbf{z}}^{\ell+1} - \bar{\mathbf{a}} + \frac{\bar{\lambda}^\ell}{\rho}}^2_2,
				$$
				update the dual variable $\bar{\lambda}^{\ell+1} = \bar{\lambda}^\ell + \rho(\bar{\mathbf{z}}^{\ell+1} - \bar{\mathbf{a}}^{\ell+1})$, and broadcast $\Pi^{\ell+1}=\bar{\mathbf{z}}^{\ell+1} - \bar{\mathbf{a}}^{\ell+1} + \bar{\lambda}^{\ell+1} / \rho$.
		\end{enumerate}
\end{algorithm}

The great advantages of distributed optimization algorithms such as ADMM are scalability and plug-and-play capability. Since Step~1 in Algorithm~\ref{alg:ADMM} can be parallelized, the number $\mathcal{I}$ of subsystems does not affect the performance. Moreover, the specific data of the individual systems mentioned above is not known to the CE, which ensures plug-and-play capability.

\subsection{ADMM yields the Solution of the Scalarized Problem}

To apply ADMM in our framework to find efficient points of~\eqref{MOP}, we first scalarize the problem
%modify it according to Subsection~\ref{subsec:MOPandMPC} and, then, scalarize the multiobjective optimization problem 
in order to obtain~\eqref{SMOP}. %
As a second preliminary step, we introduce the auxiliary variables $\mathbf{a}_i = (a_i(k),a_i(k+1),\ldots,a_i(k+N-1))^\top \in \mathbb{R}^N$, $i \in [1:\mathcal{I}]$, and use the notation $\mathbf{a} = (\mathbf{a}_1^\top \ldots \mathbf{a}_{\mathcal{I}}^\top)^\top \in \mathbb{R}^{\mathcal{I} N}$ to rewrite~\eqref{SMOP} as 
\begin{equation}\label{SMOP2}\tag{SMOP2}
	\fbox{\centering
		\begin{minipage}{0.6\textwidth}
			\begin{align*}
				\min_{(\mathbf{a},\mathbf{s})}\quad &f_\kappa(\bar{\mathbf{a}},\mathbf{s}):=\kappa g(\bar{\mathbf{a}})+(1-\kappa)h(\mathbf{s})=\kappa\norm{\bar{\mathbf{a}}-\bar{\zeta}}_2^2+(1-\kappa)\norm{\mathbf{s}}_2^2\\
				\mathrm{subject \; to} \quad & \mathbf{z}_i-\mathbf{a}_i=0 \text{, } \mathbf{z}_i\in\mathbb{D}_i \quad\forall\,i \in [1:\mathcal{I}] \quad\text{ and } \\
				&  (\bar{\mathbf{a}},\mathbf{s})\in\mathbb{S} \quad\text{ with }\quad \bar{\mathbf{a}} = \frac 1 {\mathcal{I}} \sum_{i=1}^{\mathcal{I}} \mathbf{a}_i.
			\end{align*}
	\end{minipage}}
\end{equation}

In ADMM, each prosumer first minimizes the augmented Lagrangian $\mathcal{L}_\rho:\mathbb{R}^{\mathcal{I} N} \times \mathbb{R}^{\mathcal{I} N} \times \mathbb{R}_{\geq0}^{2N} \times \mathbb{R}^{\mathcal{I} N} \to \mathbb{R}$
\begin{align}\label{NotationLagrangian}
	\mathcal{L}_\rho (\mathbf{z}, \mathbf{a}, \mathbf{s}, \lambda) = 
	\kappa g(\bar{\mathbf{a}}) + (1-\kappa) h(\mathbf{s}) + \lambda^\top(\mathbf{z} - \mathbf{a}) + \frac{\rho \mathcal{I}}{2} \norm{\bar{\mathbf{z}}^{\ell+1} - \bar{\mathbf{a}} + \frac{\bar{\lambda}}{\rho}}^2_2
\end{align}
w.r.t.\ its own power demand for some $\rho>0$. Then, the CE minimizes $\mathcal{L}_\rho$ w.r.t.\ the aggregated power demand and the auxiliary variable $\mathbf{s}$ based on the aggregated optimal power demand provided by the prosumers.

The following theorem recalls convergence properties of ADMM, see \cite[Theorem 3.1 and Subsection 3.2]{BrauFaul18}.

\begin{theorem}[Convergence of ADMM]\label{thm:ADDM}
	Let the augmented Lagrangian be defined by~\eqref{NotationLagrangian}. Then, if there exists a saddle point of the unaugmented Lagrangian $\mathcal{L}_0$, i.e. there exists some $(\mathbf{z}^\star,\mathbf{a}^\star,\mathbf{s}^\star,\lambda^\star)$ such that
$$
	\mathcal{L}_0(\mathbf{z}^\star,\mathbf{a}^\star,\mathbf{s}^\star,\lambda)\quad
	\leq\quad \mathcal{L}_0(\mathbf{z}^\star,\mathbf{a}^\star,\mathbf{s}^\star,\lambda^\star)\quad
	\leq\quad\mathcal{L}_0(\mathbf{z},\mathbf{a},\mathbf{s},\lambda^\star)
	\qquad \forall\, (\mathbf{z},\mathbf{a},\mathbf{s}) \text{ and } \lambda,
$$
	ADMM converges for all $\mathbf{z}^0 \in \mathbb{R}^{\mathcal{I} N}$ and $\Pi^0 \in \mathbb{R}^N$ in the following sense:
	\begin{enumerate}
		\item The sequence $(\mathbf{z}^\ell - \mathbf{a}^\ell)_{\ell \in \mathbb{N}}$ converges to zero.
		\item The sequence $\left((\kappa g(\bar{\mathbf{a}}^\ell) + (1-\kappa)h(\mathbf{s}^\ell)\right)_{\ell \in \mathbb{N}}$ converges to the unique minimum of~\eqref{SMOP2}.
		\item The dual variable $\lambda^\ell$ converges to the optimal dual point $\lambda^*$.
	\end{enumerate}
\end{theorem}

Firstly, we show that ADMM indeed yields the desired values.
\begin{corollary}\label{cor:SMOP3_and_SMOP2}
	For all $\kappa\in[0,1]$ the optimal values of $g(\bar{\mathbf{a}})=\norm{\bar{\mathbf{a}}-\bar{\zeta}}_2^2$ and $h(\mathbf{s})=\norm{\mathbf{s}}_2^2$ in the optimization problem
	\begin{align}
		\min_{(\mathbf{a}, \mathbf{s})} \quad & %f_\kappa (\bar{\mathbf{a}}, \mathbf{s}) =
		\kappa \norm{\bar{\mathbf{a}} - \bar{\zeta}}_2^2 + (1-\kappa) \norm{\mathbf{s}}_2^2 + \frac{\rho\mathcal{I}}{2} \norm{\bar{\mathbf{z}}^{\ell+1} - \bar{\mathbf{a}} +\frac{\bar{\lambda}}{\rho}}^2_2 \label{NotationObjectiveFunctionADMM} \\
		\mathrm{subject \; to} \quad & (\bar{\mathbf{a}}, \mathbf{s}) \in \mathbb{S} \quad \text{ and }\quad \mathbf{z}_i - \mathbf{a}_i = 0, \mathbf{z}_i \in \mathbb{D}_i \quad \forall \, i\in [1:\mathcal{I}] \nonumber
	\end{align}
	are equal to those in~\eqref{SMOP2}.
\end{corollary}
\begin{proof}
	Theorem~\ref{thm:ADDM}~(i) and~(iii) imply that the (optimal) penalty term 
	$$
		\frac{\rho\mathcal{I}}{2} \norm{\bar{\mathbf{z}}^\star - \bar{\mathbf{a}}^\star + \frac{\bar{\lambda}^\star}{\rho}}^2_2 \quad = \quad\frac{\mathcal{I}}{2\rho}\norm{\bar{\lambda}^\star}_2^2
	$$
	is constant, from which the assertion follows.
\end{proof}

Next, we work out that the optimal value of~\eqref{SMOP2} also uniquely determines the optimal values of the functions $g$ and $h$. 
\begin{proposition}\label{prop:uniqueness_of_minimizer_SOP}
	For each $\kappa \in [0,1]$ there exists an optimal solution $(\bar{\mathbf{a}}^\star,\mathbf{s}^\star)$ of Problem~\eqref{SMOP2}. For $\kappa \in (0,1)$ the solution is unique. Furthermore, for $\kappa = 0$ or $\kappa = 1$ the optimal value $h(\mathbf{s}^\star)$ or $g(\bar{\mathbf{a}}^\star)$ is unique, resp. In particular, $\mathbf{s}^\star$ is unique for $\kappa = 0$ and $\bar{\mathbf{a}}^\star$ is unique for $\kappa = 1$.
\end{proposition}

\begin{proof}
	First consider $\kappa\in(0,1)$ and the modified problem~\eqref{SMOP2}. %
We show that $f_\kappa$ is strictly convex. Since $f_\kappa$ is twice differentiable, it suffices to show that its Hessian $H_{f_\kappa}$ is positive definite, see, e.g.\ \cite[Subsection~6.4]{Luen84}. It holds
\begin{align*}
	\nabla f_\kappa(\bar{\mathbf{a}}, \mathbf{s}) = \frac{\mathrm{d}}{\mathrm{d}(\bar{\mathbf{a}}, \mathbf{s})} \left(\kappa\sum_{k=1}^N(\bar{\mathbf{a}}_k - \bar{\zeta}_k)^2 + (1-\kappa) \sum_{\ell=1}^{2N} \mathbf{s}_\ell^2\right) 
	= 2 
	\begin{pmatrix}
		\kappa(\bar{\mathbf{a}} - \bar{\zeta}) \\
		(1-\kappa)\mathbf{s}
	\end{pmatrix}
	\in \mathbb{R}^{3N}
\end{align*}
and therefore
\begin{align*}
	H_{f_\kappa} (\bar{\mathbf{a}}, \mathbf{s}) =2 \, \frac{\mathrm{d}}{\mathrm{d} (\bar{\mathbf{a}}, \mathbf{s})}
	\begin{pmatrix}
		\kappa(\bar{\mathbf{a}} - \bar{\zeta}) \\
		(1-\kappa) \mathbf{s}
	\end{pmatrix}
	=2 \, \mathrm{diag}(\underbrace{\kappa, \ldots, \kappa}_{N \text{ times}},\underbrace{1 - \kappa, \ldots, 1 - \kappa}_{2N \text{ times}}) \in \mathbb{R}^{3N \times 3N}.
\end{align*}
Furthermore, since $\kappa\in(0,1)$, the Hessian $H_{f_\kappa}$ is positive definite and thus $f_\kappa$ is strictly convex.
Hence, there is a unique minimizer $(\bar{\mathbf{a}}^\star,\mathbf{s}^\star)$ for~\eqref{SMOP2}.

Now consider $\kappa\in\{0,1\}$. Since $f_\kappa$ is at least convex in this case, there exists a unique optimal value $\kappa g(\bar{\mathbf{a}}^\star)$ or $(1-\kappa)h(\mathbf{s}^\star)$ and hence, $g(\bar{\mathbf{a}}^\star)$ or $h(\mathbf{s}^\star)$ is unique, resp. Furthermore, strict convexity of $h$ ($g$) yields uniqueness of $\mathbf{s}^\star$ for $\kappa = 0$ ($\bar{\mathbf{a}}^\star$ for $\kappa = 1$).
\end{proof}

Due to Proposition~\ref{prop:uniqueness_of_minimizer_SOP} we can reformulate~\eqref{SMOP2} as 
\begin{equation}\label{SMOP3}\tag{SMOP3}
	\fbox{
	\centering
	\begin{minipage}{0.6\textwidth}
	\begin{align*}
		\min_{(g,h)}\quad & \tilde{f}_\kappa(g,h) = \kappa g + (1-\kappa)h\\
		\mathrm{subject \; to} \quad & (g,h) \in S = \Set{(\hat{g}, \hat{h}) \in \mathbb{R}^2 | 
		\begin{matrix}
			\exists \, (\bar{\mathbf{a}}, \mathbf{s}) \in \mathbb{S} : \hat{g} = g(\bar{\mathbf{a}}), \, \hat{h} = h(\mathbf{s}), \\ 
			\exists \, \mathbf{z} \in \mathbb{D} : \bar{\mathbf{a}} = \frac {1}{\mathcal{I}} \sum_{i=1}^{\mathcal{I}} \mathbf{z}_i
		\end{matrix}},
	\end{align*}
	\end{minipage}
	}
\end{equation}
where the optimization is w.r.t. the set of feasible function value pairs.

A fundamental observation on existence and uniqueness of solutions of~\eqref{SMOP3} can be directly deduced from Proposition~\ref{prop:uniqueness_of_minimizer_SOP}.
\begin{corollary}
For each $\kappa \in [0,1]$ Problem~\eqref{SMOP3} has an optimal solution. For $\kappa \in (0,1)$ the solution is unique.
\end{corollary}

\section{Pareto Frontier}

This section is dedicated to the Pareto frontier of the optimization problem~\eqref{MOP}. %
Firstly, we characterize the Pareto frontier and provide a brief sensitivity analysis w.r.t. the initial SOC and the size of the tube constraints. Then, we quantify the trade-off between the two objectives using the concept of proper optimality.

%we calculate the optimal values of the component functions and investigate the coherence between those values and the weighting parameter. Recalling some well-known results on multiobjective optimization we are able to determine the Pareto frontier. Furthermore, we provide a brief sensitivity analysis of the Pareto frontier w.r.t. the tube constraints and the initial state of charge. In Subsection~\ref{subsec:ProperOptimality} we then strengthen the optimality concept by introducing an additional constraint, which bounds the trade-off between both objectives.

\subsection{Determination of the Pareto frontier}

We study the Pareto frontier of~\eqref{MOP} for the open-loop optimal control problem to be solved at an arbitrary but fixed time instant $k \in \mathbb{N}_0$ within the MPC scheme presented in Algorithm~\ref{alg:MPC}. We illustrate our results for $k = 0$ emphasizing that the numerical findings can be directly transferred to other time instants.

Since we scalarize \eqref{MOP} using a weighted sum, we are able to use the following characterization of efficiency given as Propositions~3.9 and~3.10 in~\cite{Ehrg05}.

\begin{proposition}\label{prop:ehrgott_pareto_opt}
Consider MOP~\eqref{OP} and its scalarization
\begin{align}\label{eq:TheoremPropOptScalarization}
	\min_{x \in \mathcal{M}} \quad \sum_{i=1}^m \mu_i f_i(x) 
\end{align}
with $\mu_i \geq 0$, $i \in [1:m]$. 
\begin{enumerate}
	\item If $\mu_j > 0$ for some $j \in [1:m]$ and $x^\star$ is an optimal solution of~\eqref{eq:TheoremPropOptScalarization}, then $x^\star$ is weakly efficient for~\eqref{OP}.
	\item If $x^\star$ is a unique optimal solution of~\eqref{eq:TheoremPropOptScalarization}, then $x^\star$ is efficient for~\eqref{OP}.
	\item Let in addition $\mathcal{M}$ be a convex set and let $f_i$, $i \in [1:m]$, be convex functions. If $x^\star$ is weakly efficient for~\eqref{OP} then there exist $\mu_i \geq 0$, $i \in [1:m]$, such that $x^\star$ is an optimal solution of~\eqref{eq:TheoremPropOptScalarization}.
\end{enumerate}
\end{proposition}

The main result of this section, a characterization of the Pareto frontier of~\eqref{MOP}, is stated in the following proposition. Note that the optimal vector $\mathbf{s}^\star$ and the optimal vector $\bar{\mathbf{z}}^\star$ are unique for $\kappa = 0$ and for $\kappa = 1$, resp. (see Proposition~\ref{prop:uniqueness_of_minimizer_SOP}). To find efficient points for $\kappa = 0$ and $\kappa = 1$, we consider the auxiliary problems
\begin{align}\label{auxi_problem1}
\begin{split}
	\min_{\bar{\mathbf{z}}} \quad & g(\bar{\mathbf{z}}) \\
	\mathrm{subject \; to} \quad & (\bar{\mathbf{z}},\mathbf{s}^\star) \in \mathbb{S}, \\ 	& \bar{\mathbf{z}} : \exists\, (\mathbf{z}_1^\top,\ldots,\mathbf{z}_{\mathcal{I}}^\top)^\top \in \mathbb{D} \text{ satisfying \eqref{NotationZbar}}
\end{split}
\end{align}
and
\begin{align}\label{auxi_problem2}
\begin{split}
	\min_{\mathbf{s}} \quad & h(\mathbf{s}) \\
	\mathrm{subject \; to} \quad & (\bar{\mathbf{z}}^\star,\mathbf{s}) \in \mathbb{S}, \hphantom{(\mathbf{z}_\mathcal{I}^\top)^top \in \mathbb{D} \text{ satisfying \eqref{NotationZbar}}}
	%\; \mathbf{z} \in \mathbb{D} 
\end{split}
\end{align}
%where $h(\mathbf{s}^\star)$ and $g(\bar{\mathbf{a}}^\star)$ denote the unique optimal value of~\eqref{SMOP2} for $\kappa = 0$ and $\kappa = 1$
resp.

\begin{proposition}\label{prop:main_result}
Let us consider the multiobjective optimization problem~\eqref{MOP}. Then, the following statements hold true:
\begin{enumerate}
	\item The set of weakly efficient points of~\eqref{MOP} is compact and connected.
	\item Consider $\kappa \in (0,1)$ and the unique optimal solution $(g^\star,h^\star)$ of~\eqref{SMOP3}. Then, there exists a unique weakly efficient point $(\bar{\mathbf{z}}^\star,\mathbf{s}^\star)$ of~\eqref{MOP} with $g(\bar{\mathbf{z}}^\star) = g^\star$ and $h(\mathbf{s}^\star) = h^\star$. Moreover, $(\bar{\mathbf{z}}^\star,\mathbf{s}^\star)$ is efficient.
	\item Consider $\kappa \in \{0,1\}$ and an optimal solution $(g^\star,h^\star)$ of~\eqref{SMOP3}. Then, there exists a unique weakly efficient point $(\bar{\mathbf{z}}^\star,\mathbf{s}^\star)$ of~\eqref{MOP} with $g(\bar{\mathbf{z}}^\star) = g^\star$ and $h(\mathbf{s}^\star) = h^\star$. If $g(\bar{\mathbf{z}}^\star)$ or $h(\mathbf{s}^\star)$ is the optimal value of Problem~\eqref{auxi_problem1} for $\kappa = 0$ or Problem~\eqref{auxi_problem2} for $\kappa = 1$, then $(\bar{\mathbf{z}}^\star,\mathbf{s}^\star)$ is efficient.
%	\item Let $\kappa \in \{0,1\}$. Then, there exists a unique Pareto optimal point $(\bar{\mathbf{z}}^\star,\mathbf{s}^\star)$ of~\eqref{MOP} with $g(\bar{\mathbf{z}}^\star) \leq g^\star$ and $h(\mathbf{s}^\star) \leq h^\star$ for all optimal solutions $(g^\star,h^\star)$ of~\eqref{SMOP3}.  
	\item The function $P : [0,1] \to \mathbb{R}^2$, $\kappa \mapsto (g(\bar{\mathbf{z}}^\star),h(\mathbf{s}^\star))$, where
	\begin{equation}
		\begin{cases}
			\quad (\bar{\mathbf{z}}^\star,\mathbf{s}^\star) \text{ is the unique optimal solution of~\eqref{SMOP}}, \quad & \text{if } \kappa \in (0,1) \\
			& \\
			\quad \mathbf{s}^\star \text{ is the unique optimal vector such that } (\bar{\mathbf{z}},\mathbf{s}^\star) \text{ is solution}  & \\	
			\quad \text{of~\eqref{SMOP} for some } \bar{\mathbf{z}} \text{ with } (\bar{\mathbf{z}},\mathbf{s}^\star) \in \mathbb{S} \text{ and }\bar{\mathbf{z}}^\star \text{ is the}\\
			\quad \text{unique optimal solution of~\eqref{auxi_problem1}}, & \text{if } \kappa = 0 \\
			& \\
			\quad \bar{\mathbf{z}}^\star \text{ is the unique optimal vector such that } (\bar{\mathbf{z}}^\star,\mathbf{s}) \text{ is solution} & \\
		\quad \text{of~\eqref{SMOP} for some } \mathbf{s} \text{ with } (\bar{\mathbf{z}}^\star,\mathbf{s}) \in \mathbb{S} \text{ and }\mathbf{s}^\star \text{ is the }\\
		\quad \text{unique optimal solution of~\eqref{auxi_problem2}}, & \text{if } \kappa = 1
		\end{cases}
		\nonumber
	\end{equation}
	is well-defined and the component $P_1$ strictly decreases whereas $P_2$ strictly increases. Furthermore, $P$ is injective and the image of $P$ coincides with the Pareto frontier of~\eqref{MOP}.
\end{enumerate}
\end{proposition}

\begin{proof}
We show the single claims of the proposition consecutively.
\begin{enumerate}
\item As a subset of the compact feasible set of~\eqref{MOP} the set of weakly efficient points is bounded. Moreover, Theorem~4.6 in~\cite{Luc88} provides closedness and connectedness of the set of weakly efficient points of~\eqref{MOP} and hence, compactness, which shows the assertion.

\item Due to construction of~\eqref{SMOP3} Statement~(ii) follows directly from Proposition~\ref{prop:ehrgott_pareto_opt}~(ii).

\item Similar to the second claim the first part of Statement~(iii) can be deduced from Proposition~\ref{prop:ehrgott_pareto_opt}~(i). To prove the second part assume without loss of generality $\kappa = 0$. Due to Proposition~\ref{prop:uniqueness_of_minimizer_SOP} there exists a minimizer $(\bar{\mathbf{a}},\mathbf{s}^\star)$ of~\eqref{SMOP2} and hence a minimizer $(\bar{\mathbf{z}},\mathbf{s}^\star)$ of~\eqref{SMOP} where $\mathbf{s}^\star$ is unique. Moreover, since $g$ is strictly convex, the optimal solution $\bar{\mathbf{z}}^\star$ of Problem~\eqref{auxi_problem1} is unique. By construction, $(\bar{\mathbf{z}}^\star,\mathbf{s}^\star)$ is efficient for~\eqref{MOP}.

\item Well-definedness of~$P$ is directly inherited from $g$ and $h$. The monotonicity properties can be seen as follows. For given $\kappa \in [0,1)$, choose $\varepsilon \in (0,1-\kappa]$. Then, we observe that $P_2(\kappa + \varepsilon) < P_2(\kappa)$ implies $P_1(\kappa + \varepsilon) > P_1(\kappa)$ and vice versa, i.e.\ if one component of the scalarized objective function strictly decreases, the other component has to strictly increase. Otherwise, this would directly yield a contradiction to the optimality of $P_2(\kappa)$ and $P_1(\kappa)$. 

We prove the assertion by contradiction. To this end, let us assume that $P_1(\kappa) < P_1(\kappa + \varepsilon)$. Due to our preliminary considerations, this implies $P_2(\kappa + \varepsilon) < P_2(\kappa)$. Next, optimality yields the inequalities
	\begin{align*}
		(1-\kappa) \Big( P_2(\kappa)-P_2(\kappa + \varepsilon) \Big) & \leq \kappa \Big( P_1(\kappa + \varepsilon) - P_1(\kappa) \Big), \\
		(\kappa+\varepsilon) \Big( P_1(\kappa + \varepsilon) - P_1(\kappa) \Big) & \leq (1-\kappa-\varepsilon) \Big( P_2(\kappa)-P_2(\kappa + \varepsilon) \Big).
	\end{align*}	
	Due to our assumption and its implication, which show that the differences in brackets are strictly positive, we get
	\begin{align*}
		\varepsilon \leq \Big( (1-\kappa) - \varepsilon \Big) \frac {P_2(\kappa)-P_2(\kappa + \varepsilon)}{P_1(\kappa + \varepsilon) - P_1(\kappa)} - \kappa \leq -\varepsilon \frac {P_2(\kappa)-P_2(\kappa + \varepsilon)}{P_1(\kappa + \varepsilon) - P_1(\kappa)} < 0,
	\end{align*}
	a contradiction to $\varepsilon > 0$.
	
	Note that these considerations also imply the assertion for $\kappa = 1$ since the case $P_2(1) < P_2(\kappa)$ for some $\kappa \in [0,1)$ is equivalent to $P_2(\kappa + \varepsilon) < P_2(\kappa)$, where $\varepsilon = 1 - \kappa$. This yields $P_1(\kappa + \varepsilon) > P_1(\kappa)$ and hence, $P_1(1) > P_1(\kappa)$, a contradiction to the optimality of $P_1(\kappa)$.
	
Since the components $P_1$ and $P_2$ are strictly monotone, $P$ is injective. Furthermore, Proposition~\ref{prop:ehrgott_pareto_opt}~(iii) yields that each efficient point $(\bar{\mathbf{z}}^\star,\mathbf{s}^\star)$ of~\eqref{MOP} can be found by varying $\kappa \in [0,1]$ and solving~\eqref{SMOP}. Hence, by construction,  
\begin{align}
	P([0,1]) = \Set{(g(\bar{\mathbf{z}}^\star),h(\mathbf{s}^\star)) | (\bar{\mathbf{z}}^\star,\mathbf{s}^\star) \text{ is Pareto optimal of~\eqref{MOP}}}, \nonumber
\end{align}
which completes the proof.
\end{enumerate}
\end{proof}

To determine the Pareto frontier numerically we run Algorithm~\ref{alg:ADMM} for different $\kappa \in [0,1]$. Here, we set $\ubar{\mathbf{c}}_i$ and $\bar{\mathbf{c}}_i$, $i\in[1:\mathcal{I}]$, according to the values depicted in Figure~\ref{fig:pareto} (top right). Let $\mathbf{z}_i(k) \in \mathbb{D}_i$, $i\in[1:\mathcal{I}]$, and $\Pi^0 \in \mathbb{R}^N$ be given.  For given $\kappa\in[0,1]$, the unique optimal values 
%$g(\bar{\mathbf{a}}^\star)$ and $h(\mathbf{s}^\star)$ 
$P_1(\kappa)$ and $P_2(\kappa)$ are denoted by $g_\kappa$ and $h_\kappa$, resp. The results can be found in Figure~\ref{tab:pareto}\hspace{-1.5mm}. Note that for $\kappa = 0.6$ the maximum of $f_\kappa$ is obtained.
\begin{figure}[htb]
	\begin{minipage}{0.525\textwidth}
		\centering\vspace*{3.375cm}
		\includegraphics[width=1.\textwidth]{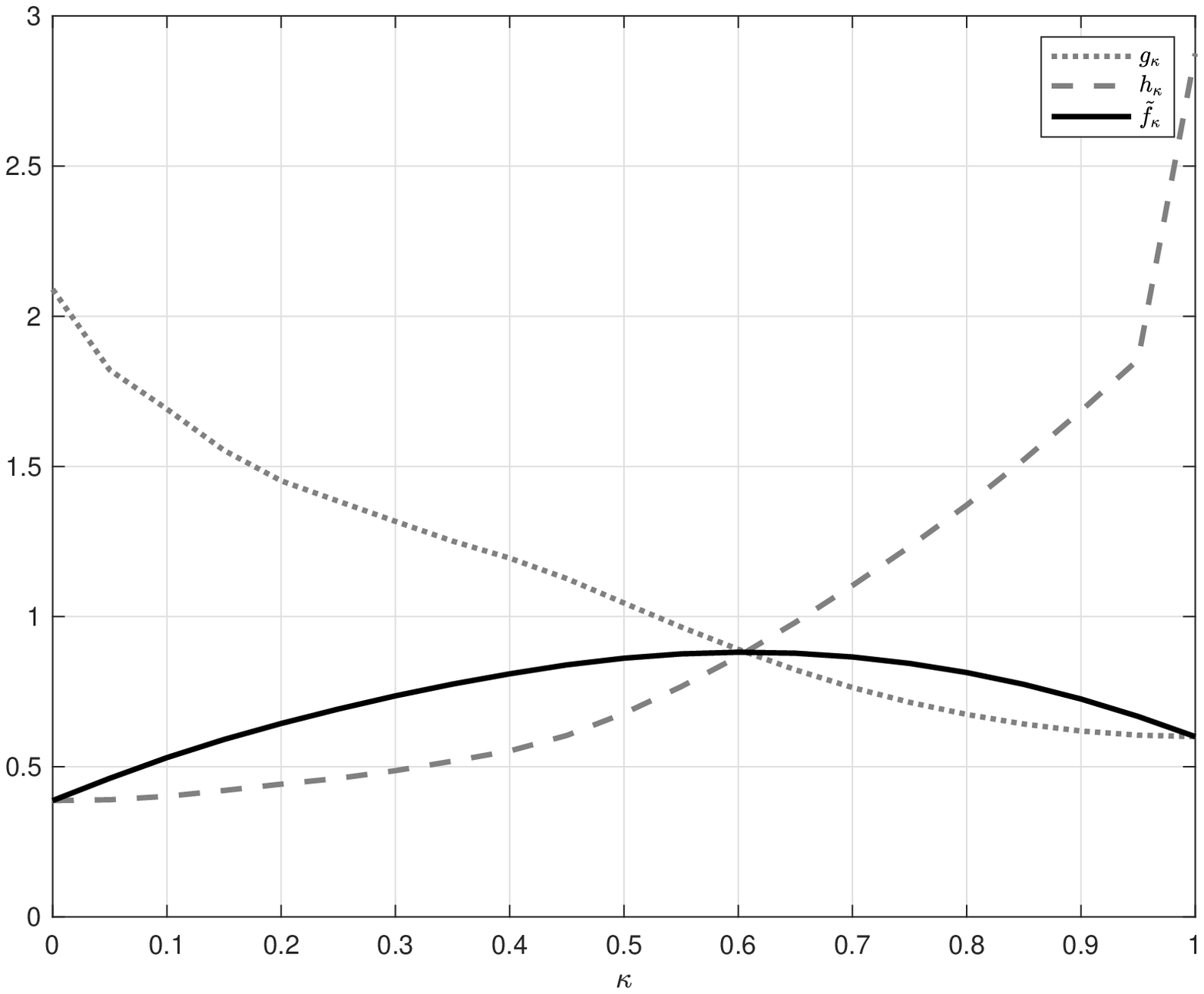}
	\end{minipage}
\begin{minipage}{0.45\textwidth}
	\centering
	\begin{tabular}{|c||c|c|c|}\hline
		$\kappa$ & $g_\kappa$ & $h_\kappa$ & $\tilde{f}_\kappa(g_\kappa,h_\kappa)$ %& ADMM~\eqref{NotationObjectiveFunctionADMM} %$\tilde{f}_\kappa(g_\kappa,h_\kappa)+\frac{\rho\I}{2}\norm{\bar{\mathbf{z}}^\star-\bar{\mathbf{a}}^\star+\frac{\bar{\lambda}^\star}{\rho}}_2^2$
		\\ \hline\hline
		0.000 & 2.090 & 0.387  &       0.387 \\ 
		0.050 & 1.821 & 0.390  &       0.462 \\ 
		0.100 & 1.691 & 0.401  &       0.530 \\ 
		0.150 & 1.553 & 0.421  &       0.591 \\ 
		0.200 & 1.452 & 0.442  &       0.644 \\ 
		0.250 & 1.384 & 0.462  &       0.692 \\ 
		0.300 & 1.317 & 0.487  &       0.736 \\ 
		0.350 & 1.251 & 0.519  &       0.775 \\ 
		0.400 & 1.195 & 0.553  &       0.809 \\ 
		0.450 & 1.127 & 0.604  &       0.839 \\ 
		0.500 & 1.045 & 0.678  &       0.862 \\ 
		0.550 & 0.965 & 0.767  &       0.876 \\ 
		0.600 & 0.891 & 0.867  &       0.881 \\ 
		0.650 & 0.823 & 0.981  &       0.878 \\ 
		0.700 & 0.764 & 1.104  &       0.866 \\ 
		0.750 & 0.715 & 1.233  &       0.844 \\ 
		0.800 & 0.674 & 1.372  &       0.814 \\ 
		0.850 & 0.642 & 1.522  &       0.774 \\ 
		0.900 & 0.619 & 1.685  &       0.726 \\ 
		0.950 & 0.605 & 1.855  &       0.668 \\ 
		1.000 & 0.600 & 2.874  &       0.600 \\ \hline
	\end{tabular}
	\end{minipage}
%	\hspace{2mm}
%	\begin{minipage}{0.525\textwidth}
%		\centering\vspace*{3.375cm}
%		\includegraphics[width=1.\textwidth]{figures/g_and_h}
%	\end{minipage}
	\caption{Optimal values $g_\kappa$ and $h_\kappa$ and the scalarized objective function $\tilde{f}_\kappa$ introduced in~\eqref{SMOP3} 
	 for all $\kappa \in 0.05 \cdot [0:20]$.}
	\label{tab:pareto}
\end{figure}

Next, we display the values in the table in Figure~\ref{tab:pareto} in a $(g-h)$-plane to get an approximation of the Pareto frontier, which is visualized in Figure~\ref{fig:pareto} (top left). Note that for $\kappa \to 0$ and $\kappa \to 1$ the arcs asymptotically go to the minimum value of the optimization problem $\min_{(\bar{\mathbf{a}},\mathbf{s}) \in \mathbb{S}} \, h(\mathbf{s})$ and $\min_{(\bar{\mathbf{a}},\mathbf{s}) \in \mathbb{S}} \, g(\bar{\mathbf{a}})$, resp.
\begin{figure}[htb]
	\centering
	\includegraphics[width=0.45\textwidth]{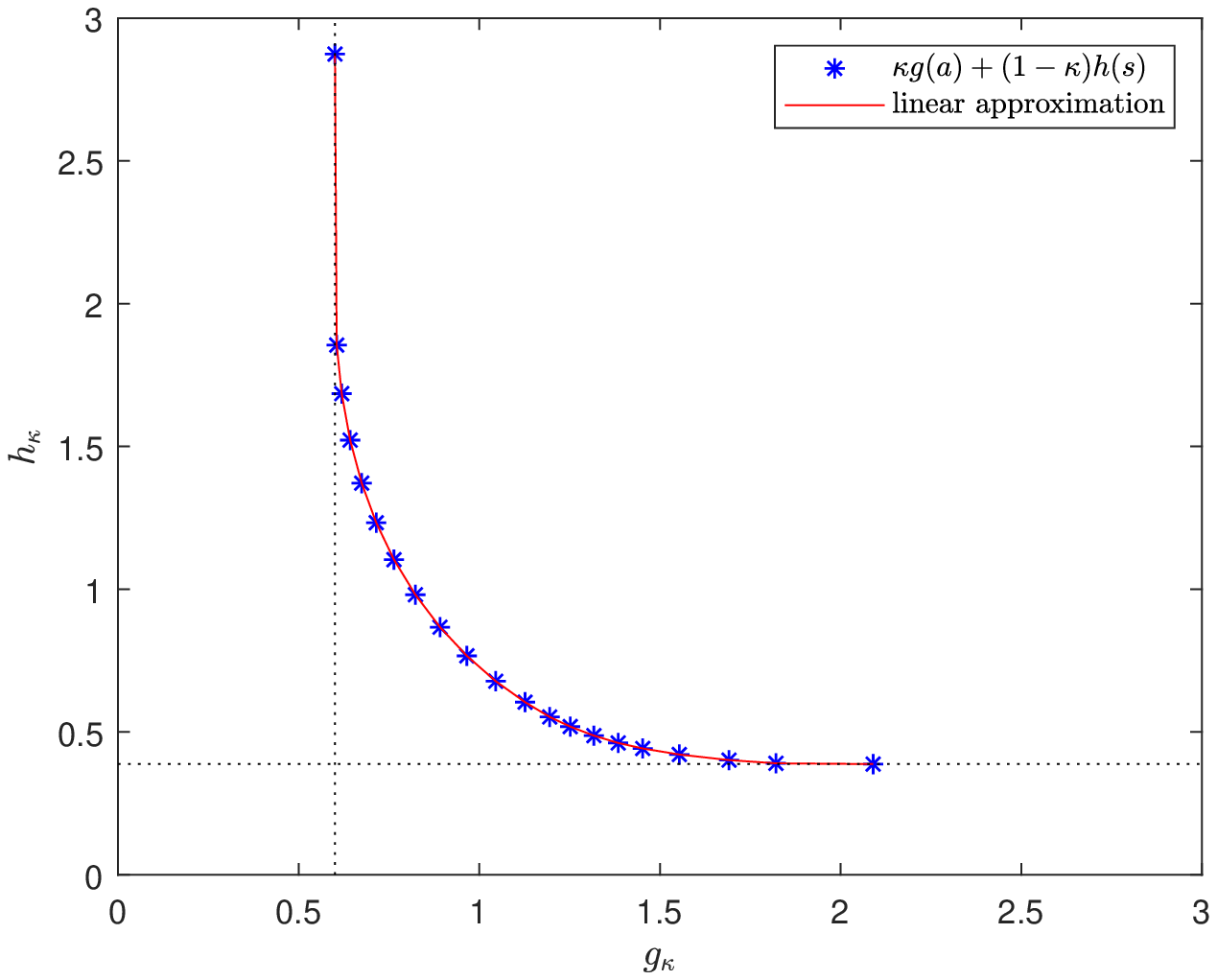}\hspace{8mm}
	\includegraphics[width=0.43\textwidth]{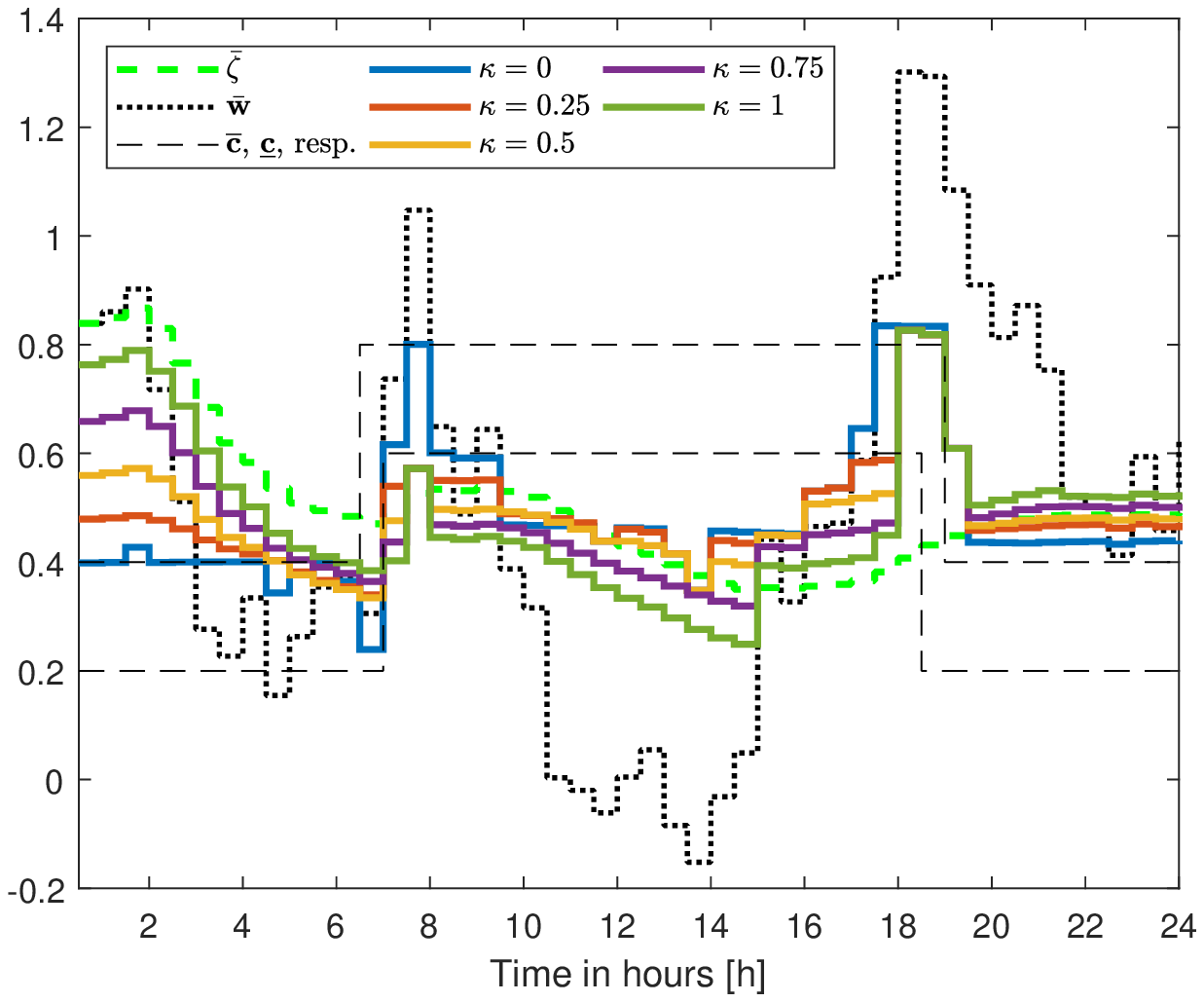}\\[2mm]
	\includegraphics[width=0.45\textwidth]{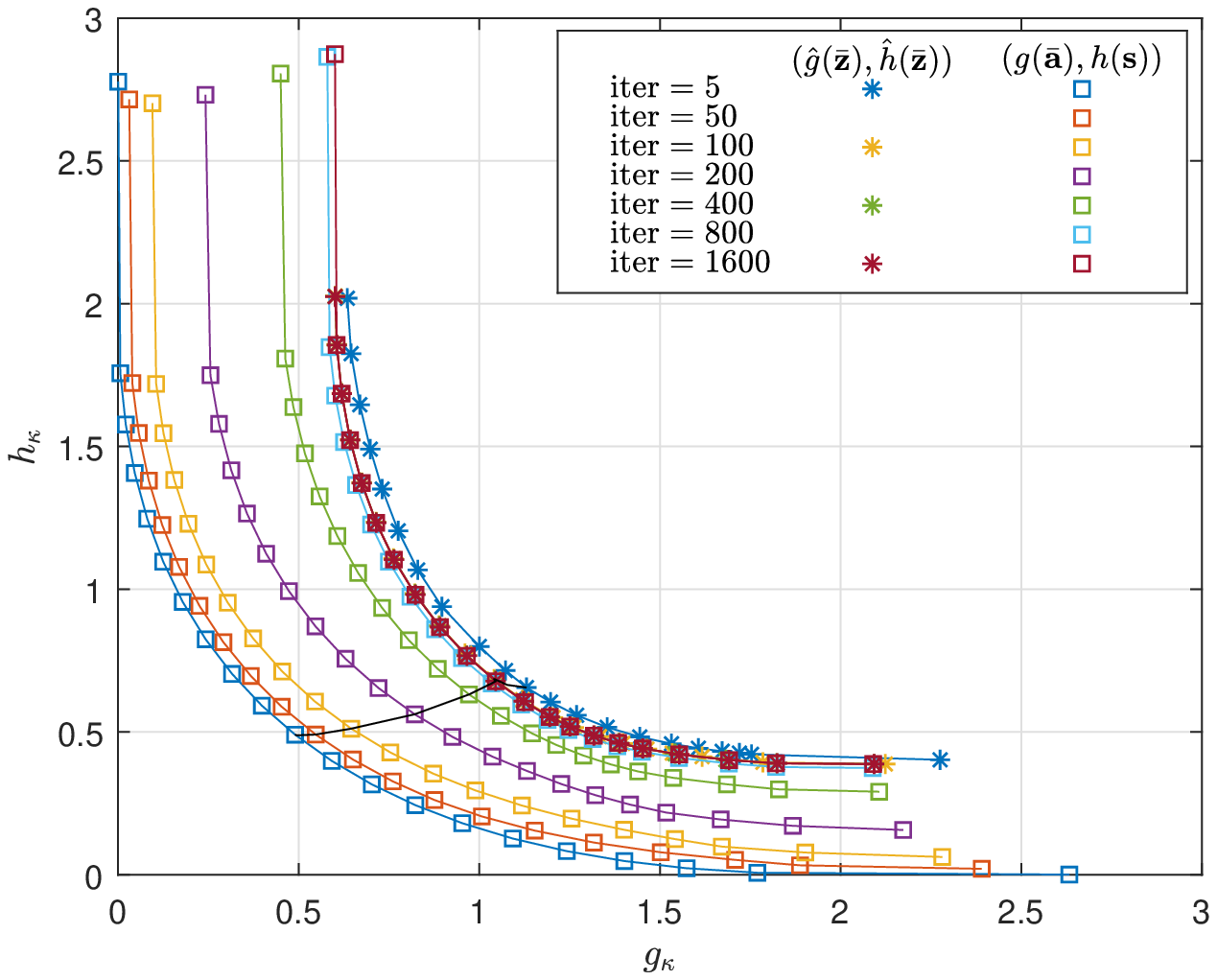}\hspace{5mm}
	\includegraphics[width=0.45\textwidth]{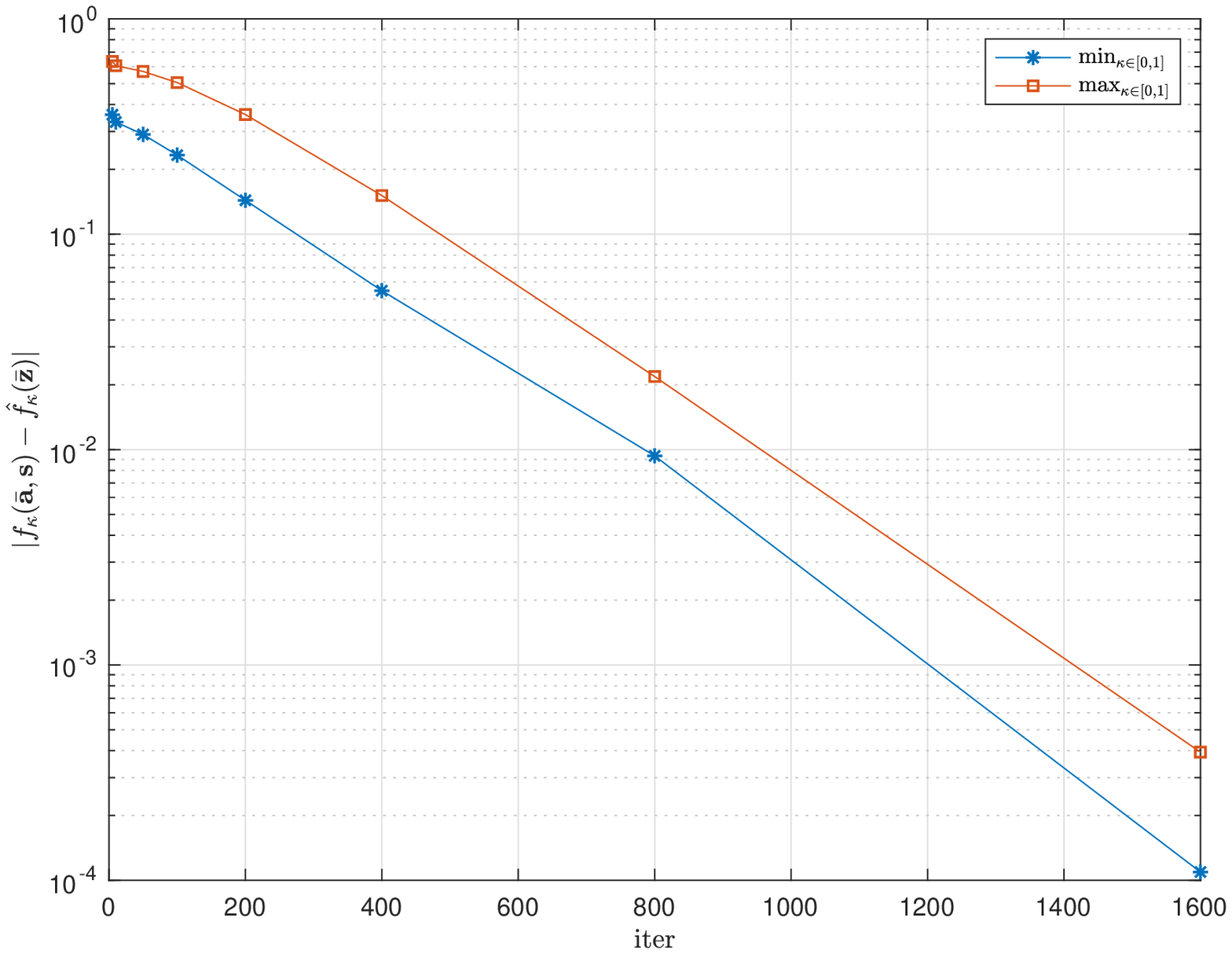}
	\caption{Pareto frontier in a $(g-h)$-plane (top left), open-loop perfomance (top right), convergence of the Pareto frontier (bottom left), and convergence of the cost residual $|f_\kappa - \hat{f}_\kappa|$ (bottom right). The solid black line (bottom left) connects all solutions corresponding to $\kappa = 0.5$.}
	\label{fig:pareto}
\end{figure}
The convergence of the Pareto frontier in dependence of number of iteration steps within the ADMM-algorithm~\ref{alg:ADMM} is depicted in Figure~\ref{fig:pareto} (bottom left). Here, we consider the functions $g : \mathbb{R}^N \to \mathbb{R}$ and $h : \mathbb{R}^{2N} \to \mathbb{R}$ as defined above and in addition we introduce functions $\hat{g}, \hat{h}, \hat{f}_\kappa : \mathbb{R}^N \to \mathbb{R}$ with
\begin{align}
	\hat{g}(\bar{\mathbf{z}}) \quad & = \quad \norm{\bar{\mathbf{z}} - \bar{\zeta}}_2^2, \nonumber \\
	\hat{h}(\bar{\mathbf{z}}) \quad & = \quad
	\norm{
	\begin{pmatrix}
		\max \{\bar{\mathbf{z}} - \bar{\mathbf{c}},0\} \\
		\max \{\ubar{\mathbf{c}} - \bar{\mathbf{z}},0\}
	\end{pmatrix}
	}_2^2, \nonumber \\
	\hat{f}_\kappa(\bar{\mathbf{z}}) \quad & = \kappa \hat{g}(\bar{\mathbf{z}}) + (1-\kappa) \hat{h}(\bar{\mathbf{z}}).
\end{align}
Due to Theorem~\ref{thm:ADDM} the sequence $(\bar{\mathbf{z}}^{\ell},\mathbf{a}^\ell,\mathbf{s}^\ell)_{\ell \in \mathbb{N}_0}$ converges to the optimal solution $(\bar{\mathbf{z}}^\star,\mathbf{a}^\star,\mathbf{s}^\star)$ of~\eqref{SMOP2} fulfilling $g(\bar{\mathbf{a}}^\star) = \hat{g}(\bar{\mathbf{z}}^\star)$ and $h(\mathbf{s}^\star) = \hat{h}(\bar{\mathbf{z}}^\star)$. Regarding Figure~\ref{fig:pareto} (bottom left), however, it can be observed that the convergence rate of $\bar{\mathbf{z}}^\ell \to \bar{\mathbf{z}}^\star$ is much higher than the convergence rate of $(\bar{\mathbf{a}}^\ell,\mathbf{s}^\ell) \to (\bar{\mathbf{a}}^\star,\mathbf{s}^\star)$. Note that the points generated by $g$ and $h$ (before convergence) are not feasible while those generated by $\hat{g}$ and $\hat{h}$ are not optimal. 

The remaining graphic of Figure~\ref{fig:pareto} (right) shows the open-loop perfomance (top) and the behaviour of the residual $|f_\kappa - \hat{f}_\kappa|$ depending on both the weighting parameter $\kappa \in [0,1]$ and the number of maximal iterations (bottom).

\begin{remark}\label{rem:image_Pareto_as_graph}
Note that, due to Proposition~\ref{prop:main_result}~(v) the Pareto frontier as depicted in Figure~\ref{fig:pareto} (top left) can be regarded as a the graph of the function $H : P_1([0,1]) \to P_2([0,1])$, $H(P_1(\kappa)) = P_2(\kappa)$, where $P$ is defined as in Proposition~\ref{prop:main_result}.
\end{remark}

\subsection{Sensitivity Analysis}

In Figure~\ref{fig:pareto} we set the average initial SOC $\bar{x}(0) = 1.0$ and the tube size $\bar{\mathbf{c}}-\ubar{\mathbf{c}} = 0.2$. Next we explore how the Pareto frontier changes if we modify these parameters. 

In Figure~\ref{fig:impact_tube_size} (left) different Pareto frontiers depending on the tube size are plotted. Increasing the size of the tube equals a relaxation of the subproblem $\min h(\mathbf{s})$. The larger the difference $\bar{\mathbf{c}}-\underline{\mathbf{c}}$, the easier it gets to keep $h_\kappa$ small. Small tube sizes do not affect the performance of $g$ drastically. If the tube becomes big enough to keep $h(\mathbf{s})$ small without effort, however, the focus of the optimization shifts to $g(\bar{\mathbf{a}})$.

The impact of the initial SOC $\bar{x}(0)$ on the Pareto frontier is visualized in Figure~\ref{fig:impact_tube_size} (right). At first sight, the performance w.r.t objective $g$ seems to improve with increasing $\bar{x}(0)$ while the performance w.r.t. $h$ deteriorates and vice versa. To interpret this behaviour one has to take the time series $\bar{w}(n)$, $n \in [k:k+N-1]$, into acount, see Figure~\ref{fig:pareto} (top right, dotted, black line). If the batteries are completely discharged at the beginning of the time interval there is no possibility to reduce the energy demand by discharging the batteries. Hence, the tube constraints are violated due to $\bar{\mathbf{w}}(n) > \bar{c}(n)$ for small $n$. To compensate $\sum_{n=k}^{k+15} \bar{\zeta}(n) > \sum_{n=k}^{k+15} \bar{\mathbf{w}}(n)$, on the other hand, one needs to charge the batteries. Hence, the higher the SOC at the beginning, the harder it gets to trace the desired trajectory.

\begin{figure}[htb]
	\centering
	\includegraphics[width = 0.45\textwidth]{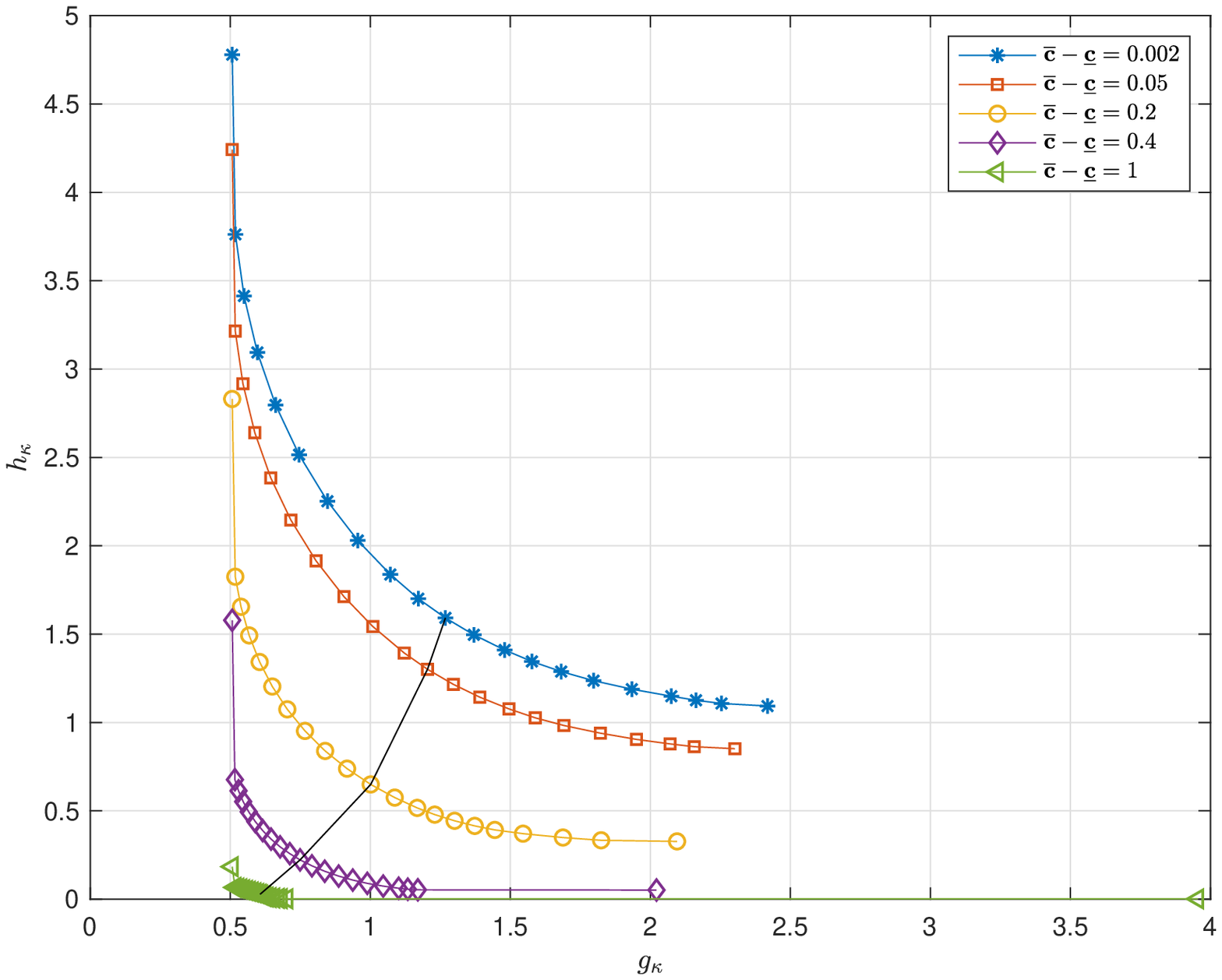}\hspace{5mm}
	\includegraphics[width = 0.45\textwidth]{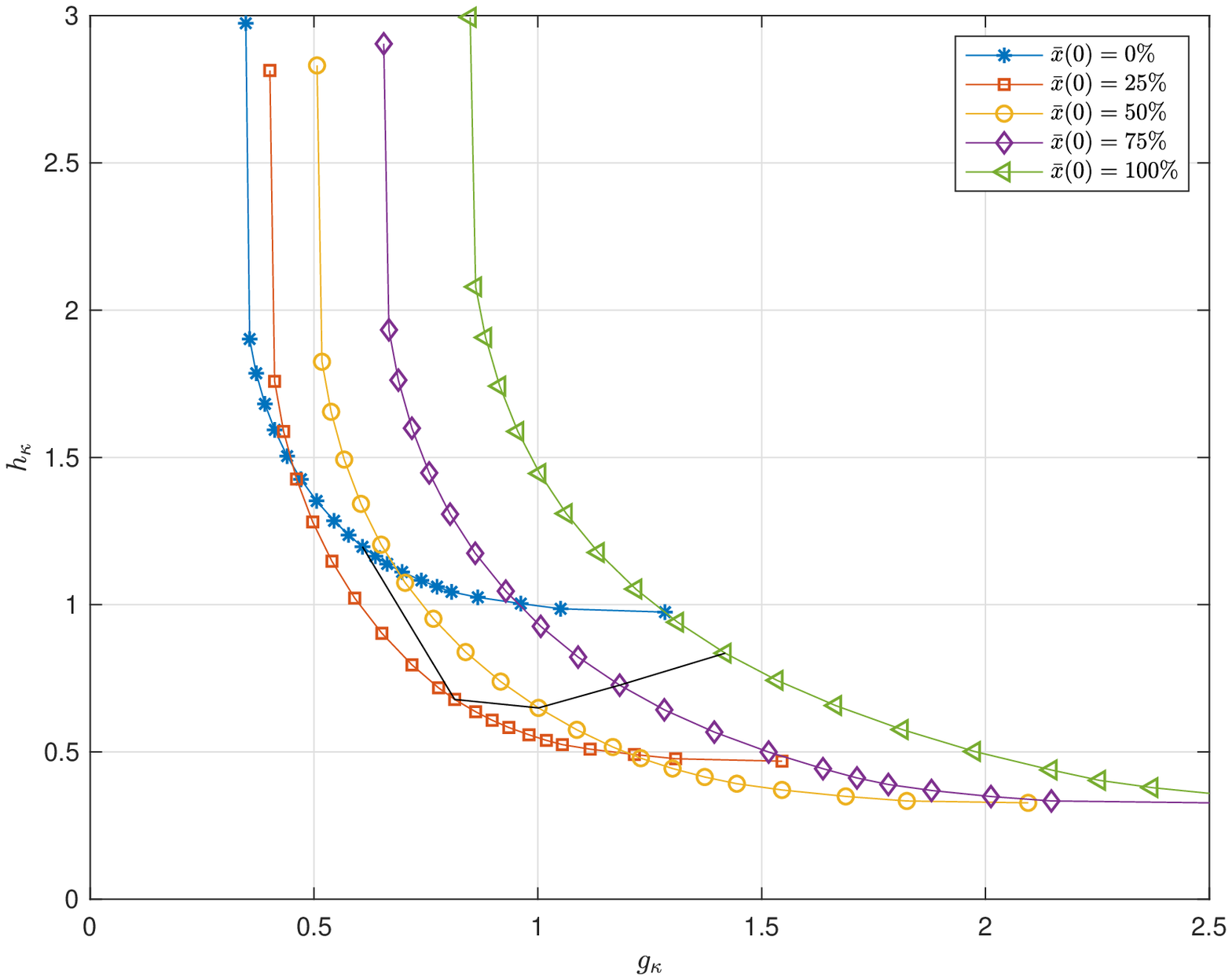}
	\caption{Visualisation of the impact of the tube constraints (left) and the initial SOC (right) on the Pareto frontier for $N=48$, $\mathrm{iter}=500$, $\rho=10^{-3}$, and $\kappa \in 0.05 \cdot [0:20]$. The solid black line connects all solutions corresponding to $\kappa = 0.5$.}
	\label{fig:impact_tube_size}
\end{figure}

\subsection{Proper Optimality}\label{subsec:ProperOptimality}

Considering an efficient point, improving the performance w.r.t. one objective is only possible at the expense of the performance w.r.t. the other. We investigate the trade-off between (potentially) conflicting objectives using the concept of proper optimality. Note that there are different definitions given by several authors~\cite{Bens83,SawaNaka85,Ehrg05}. In this paper, however, we focus on proper optimality in the sense of Geoffrion, which expands efficiency by introducing an upper bound on the trade-off between two objectives. 

\begin{definition}[Proper efficiency]\label{def:proper_min}
A point $x^\star \in \mathcal{M}$ is called \emph{properly efficient (in the sense of Geoffrion)} of the MOP~\eqref{OP} if $x^\star$ is efficient of~\eqref{OP} and there exists some $L > 0$ such that for all $i \in \{1,\ldots,m\}$ and $x \in \mathcal{M}$ with $f_i(x) < f_i(x^\star)$ there exists some index $j \in \{1,\ldots,m\}$ with $f_j(x^\star) < f_j(x)$ and
\begin{align}\label{eq:Geoffrion_constraint}
	\frac{f_i(x^\star) - f_i(x)}{f_j(x) - f_j(x^\star)} \quad \leq \quad L.
\end{align}
\end{definition}

\begin{remark}
In Figure~\ref{fig:proper_optimality} (top right) we consider the optimal solution $(g_{\kappa_0},h_{\kappa_0})$ of~\eqref{SMOP3} corresponding to a given $\kappa_0 \in [0,1]$. If we want to improve the perfomance w.r.t. objective $g$ by the amount of $g_{\kappa_0} - g_{\kappa_1}$, the performance w.r.t. objective $h$ worsens by the amount of $h_{\kappa_1} - h_{\kappa_0}$. Using function $H$ as defined in Remark~\ref{rem:image_Pareto_as_graph} and the relation $g_\kappa = P_1(\kappa)$, $\kappa \in [0,1]$, Inequality~\eqref{eq:Geoffrion_constraint} can be written as
\begin{align}
	0 \quad < \quad \frac{H(g_{\kappa_1}) - H(g_{\kappa_0})}{g_{\kappa_0} - g_{\kappa_1}} \quad \leq \quad L \quad = \quad L_{\kappa_0} \qquad \forall \, \kappa_1 \in (\kappa_0,1]. \nonumber
\end{align}
For the particular case $g_{\kappa_1} \to g_{\kappa_0}$ this reads as
\begin{align}
	 \lim_{g_{\kappa_1} \to g_{\kappa_0}} \frac{H(g_{\kappa_1}) - H(g_{\kappa_0})}{g_{\kappa_0} - g_{\kappa_1}} \quad \leq \quad L, \nonumber
\end{align}
where the left hand side is the negative of derivative of $H$ w.r.t. $g_\kappa$ at $g_{\kappa_0}$ if existent. The latter indicates the cost of an infinitesimal improvement w.r.t. the performance of $g$. Hence, proper optimality provides an upper bound to this cost which is an essential information for a decision maker.
\end{remark}

The following theorem provides a characterization of the properly efficient points of~\eqref{OP} using a weighted sum scalarization as in Proposition~\ref{prop:ehrgott_pareto_opt}.

\begin{theorem}[\cite{Ehrg05}, Theorem~3.15]
Consider MOP~\eqref{OP} with a convex feasible set $\mathcal{M}$ and convex component functions $f_i$, $i \in [1:m]$. Then the following holds true: A point $x^\star \in \mathcal{M}$ is properly efficient for~\eqref{OP} if and only if $x^\star$ is an optimal solution of~\eqref{eq:TheoremPropOptScalarization}
with positive weighting parameters $\mu_i > 0$, $i \in [1:m]$.
\end{theorem}

This result can be directly applied to~\eqref{MOP}.

\begin{corollary}
All efficient points of~\eqref{MOP} except for the extremal points $(\kappa\in\{0,1\})$ are properly efficient.
\end{corollary}

We conclude this section with an example of calculating the bound $L$ for $\kappa = 0.2$ numerically. 
\begin{example}
Consider $\kappa_0 = 0.2$. We provide an intuitive way to determine an upper bound $L$ at the corresponding optimal value pair $(g_{\kappa_0},h_{\kappa_0}) = (1.452,0.442)$. Since $g_\kappa$ and $h_\kappa$ result from an optimization routine for which there is no explicit formula, we cannot compute $L$ explicitly.

Instead, we calculate 
\begin{align}
	L_{\kappa_0}(\kappa_1) := 
	\begin{cases}
		\frac{g_{\kappa_0} - g_{\kappa_1}}{h_{\kappa_1} - h_{\kappa_0}} & \text{if} \quad g_{\kappa_1} < g_{\kappa_0} \\
		\frac{h_{\kappa_0} - h_{\kappa_1}}{g_{\kappa_1} - g_{\kappa_0}} \, = \, \left(\frac{g_{\kappa_0} - g_{\kappa_1}}{h_{\kappa_1} - h_{\kappa_0}}\right)^{-1} & \text{if} \quad g_{\kappa_0} < g_{\kappa_1} \\
		0 & \text{else}
	\end{cases}
	\nonumber
\end{align}
for all $\kappa_1$ in the table in Figure~\ref{tab:pareto}~(left). Additionally we compute $L_{\kappa_0}(0.01)$ and $L_{\kappa_0}(0.99)$ to illustrate the behaviour near the extrema. Then, we choose $\max \{L_{\kappa_0}(\kappa_1)|\kappa_1 \in K\}$, where $K = 0.05 \cdot [0:20] \cup \{0.01,0.99\}$, as an approximation of the upper bound $L_{\kappa_0}$ in~\eqref{eq:Geoffrion_constraint}. The results can be found in the table in Figure~\ref{fig:proper_optimality}.

One can observe that $L_{0.2} \approx 4.74$ holds, i.e., to gain one quantity in $g$ direction one has to spend at maximum 4.74 quantities in $h$ direction and vice versa. In our case, i.e. $\kappa_0 = 0.2$, the bound on the second relation is active while the first one can replaced by $1/L_{0.2} = 0.211$, i.e., to gain one quantity in $g$ direction one has to spend at maximum 0.211 units in $h$ direction.

\noindent In Figure~\ref{fig:proper_optimality} (bottom right) the evolution of the bound $L_\kappa$ is illustrated depending for $\kappa \in [0.05,0.99]$. 
\begin{figure}[htb]
\begin{minipage}{0.45\textwidth}
	\centering
	\vspace{8mm}
	\begin{tabular}{|c||c|c|c|}\hline
		$\kappa_1$ & $g_{\kappa_1}$ & $h_{\kappa_1}$ & $L_{0.2}(\kappa_1)$
		\\ \hline\hline
		0.000 &    2.09  &     0.39 &        0.06 \\  
		0.010 &    1.91  &     0.39 &        0.10 \\  
		0.050 &    1.82  &     0.39 &        0.12 \\  
		0.100 &    1.69  &     0.40 &        0.15 \\  
		0.150 &    1.54  &     0.42 &        0.00 \\  
		0.200 &    1.45  &     0.44 &        {\bf 4.74} \\  
		0.250 &    1.38  &     0.46 &        4.09 \\  
		0.300 &    1.32  &     0.49 &        3.53 \\  
		0.350 &    1.25  &     0.52 &        3.06 \\  
		0.400 &    1.20  &     0.55 &        2.69 \\  
		0.450 &    1.13  &     0.61 &        2.29 \\  
		0.500 &    1.04  &     0.68 &        1.95 \\  
		0.550 &    0.96  &     0.77 &        1.68 \\  
		0.600 &    0.89  &     0.87 &        1.47 \\  
		0.650 &    0.82  &     0.98 &        1.29 \\  
		0.700 &    0.76  &     1.10 &        1.15 \\  
		0.750 &    0.71  &     1.23 &        1.02 \\  
		0.800 &    0.67  &     1.37 &        0.92 \\  
		0.850 &    0.64  &     1.52 &        0.82 \\  
		0.900 &    0.62  &     1.69 &        0.73 \\  
		0.950 &    0.61  &     1.86 &        0.66 \\  
		0.990 &    0.60  &     1.99 &        0.60 \\  
		1.000 &    0.60  &     2.03 &        0.59 \\  \hline
	\end{tabular}
\end{minipage}
\begin{minipage}{0.525\textwidth}
	\centering
	\includegraphics[scale=1]{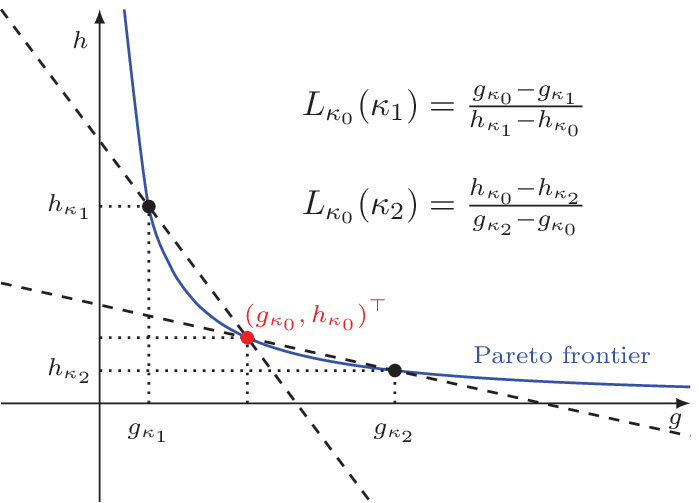}\\[3mm]
	\includegraphics[scale=0.4]{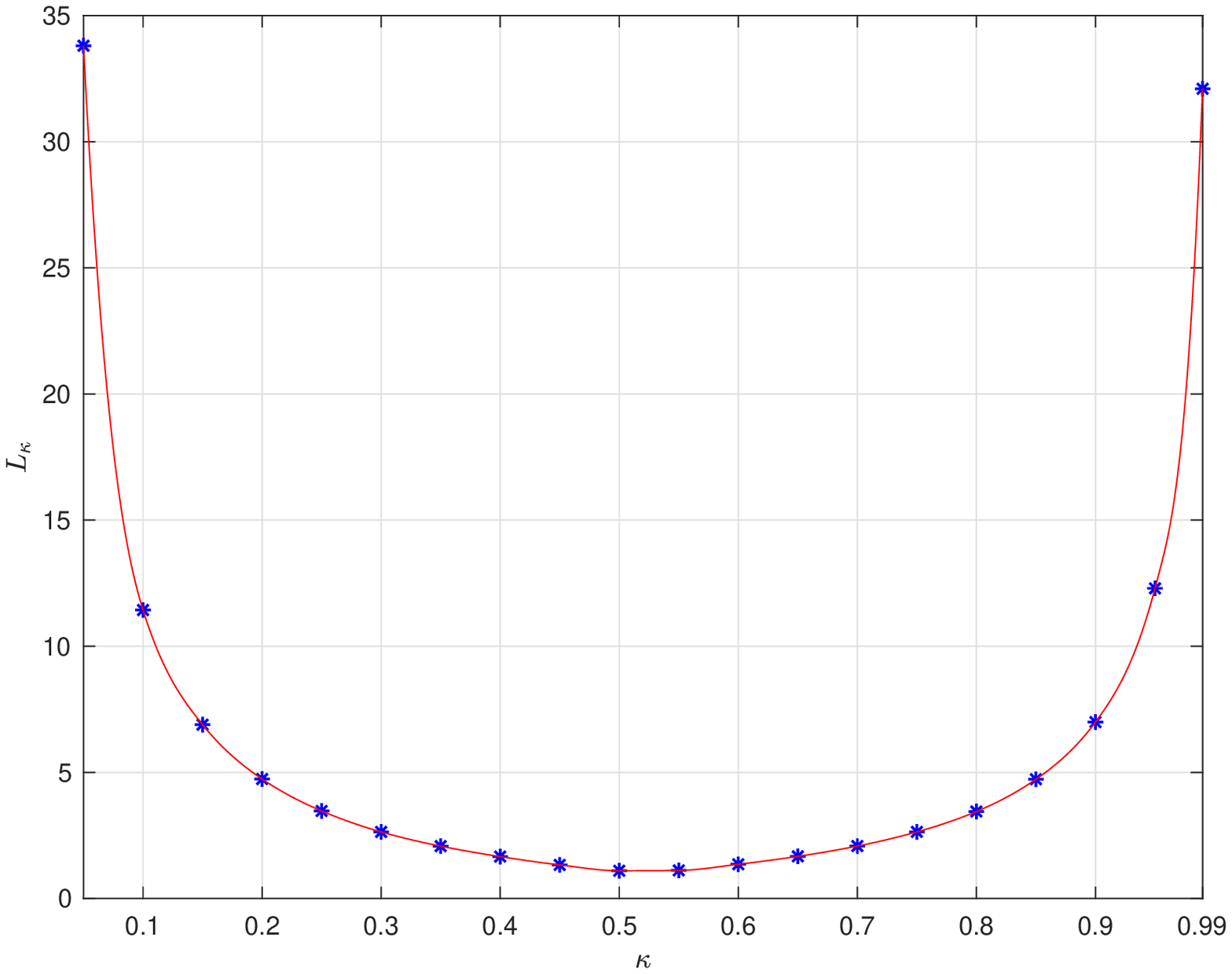}
\end{minipage}
\caption{Calculation of the bound $L_{0.2}$ in~\eqref{eq:Geoffrion_constraint} by varying $\kappa_1 \in [0,1]$ (left) and a visualization of the bound $L$ as the maximal absolute value of the differnce quotients (top right) and of the dependency of $L_\kappa$ from $\kappa \in [0.05,0.99]$ (bottom right).}
\label{fig:proper_optimality}
\end{figure}
\end{example}

\section{Conclusions and outlook}

In this paper, we have analysed the Pareto frontier of a multiobjective optimization problem, in which a trade-off between peak shaving and providing flexibility has to be made. To this end, we first considered the corresponding scalarization, linked it to ADMM in order to (numerically) compute its solution for a given scalarization parameter, and rigorously showed that by doing so we get the whole Pareto frontier. Moreover, by using the concept of proper optimality introduced by Geoffrion, we further investigated the Pareto frontier, which allowed to quantify the trade-off between the conflicting objectives.

There are several ways to extend our model: One could easily consider the impact of controllable loads as introduced in~\cite{BrauGrue16} or varying lower/upper bounds on (dis)charging rates. Another possible modification of our optimization problem is to additionally penalize the use of the batteries, e.g. in form of $\norm{u_i}$. The presented approach, especially ADMM still works, as long as the sets $\mathbb{D}_i$ of feasible solutions $\mathbf{z}_i$, $i \in [1:\mathcal{I}]$, stay polyhedral.

Throughout this paper, we assumed the net consumption to be given. Finding a suitable prediction method is an interesting topic for future research. For the goal of load shaping, the robustness of MPC schemes w.r.t. inaccurate forecasts was numerically investigated in~\cite{WortKell15}. Moreover, there are interesting approaches to estimate the energy generation within the prediction window more thoroughly using artificial neural networks as pointed out in~\cite{MabeFern08,ChowLee12}. A good starting point would be~\cite{FaulEnge18} and the references therein.

\bibliographystyle{unsrt}
\bibliography{mybibfile}

\end{document}